\documentclass[12pt]{amsart}
\usepackage{amsthm,graphicx,mathrsfs,amsmath,mathtools,amssymb}
\usepackage[all]{xy}
\usepackage[symbol]{footmisc}
\usepackage{geometry}
\usepackage{float}
\usepackage{hyperref}

\makeatletter

\def\subsection{\@startsection{subsection}{2}
  \z@{.5\linespacing}{.5\linespacing}
  {\normalfont\bfseries}}
\makeatother
\makeatletter

\renewcommand{\l@subsection}{\@tocline{2}{0pt}{2.5em}{3.5em}{}}
\makeatother
\makeatletter
\newlength{\abstractwidth}
\renewenvironment{abstract}{
  \ifx\maketitle\relax
    \ClassWarningNoLine{amsart}{Abstract should precede \protect\maketitle}
  \fi
  \global\setbox\abstractbox=\vbox\bgroup
  \normalfont
  \vspace*{1.5\baselineskip}
  \begin{center}\textsc{Abstract}\end{center}
  \vspace{0.5\baselineskip}
  \begin{center}\begin{minipage}{\abstractwidth}
  \normalfont\normalsize
  \noindent
}{
  \par\end{minipage}\end{center}
  \egroup
  \ifx\@setabstract\relax \@setabstracta \fi
}
\makeatother
\makeatletter
\renewcommand{\subsubsection}{\@startsection{subsubsection}{3}
  \z@{1.0\linespacing \@plus .7\linespacing}{.5\linespacing}
  {\normalfont\itshape}}
\makeatother
\hypersetup{
    colorlinks=true,
    linkcolor=blue,
    filecolor=magenta,
    urlcolor=cyan,
    citecolor=blue,
}

\newtheorem*{theorem*}{Theorem}
\newtheorem*{proposition*}{Proposition}
\newtheorem{theorem}{Theorem}[section]
\newtheorem{corollary}[theorem]{Corollary}
\newtheorem{lemma}[theorem]{Lemma}
\newtheorem{proposition}[theorem]{Proposition}
\theoremstyle{definition}
\newtheorem{definition}[theorem]{Definition}
\newtheorem{rem}[theorem]{Remark}
\theoremstyle{remark}
\theoremstyle{definition}
\newtheorem{example}[theorem]{Example}
\def\sus{\Sigma}
\def\suse{\mathbf{\Sigma}}
\def\lup{\Omega}
\def\eunder{\mathbf{E}^{\mathbb{R}_+}}
\newcommand\lupe{\mathbf{\Omega}}
\def\quic{{\mathscr C}}
\def\quid{{\mathscr D}}
\def\br{\mathbb{R}}
\def\ray{{\mathbb{R}_+}}
\def\E{\mathbf{E}}
\def\N{\mathbb{N}}
\def\Z{\mathbb{Z}}
\newcommand{\adj}{{\scriptscriptstyle\vee}}
\def\im{{\rm Im\,}}
\def\bsph{\mathcal{S}}
\def\bsphc{\mathcal{S}_\ray}
\def\bdisk{\mathcal{D}}
\def\bcil{\mathcal{C}}
\newcommand{\spec}{{{\bf Sp}}}
\newcommand{\spece}{{{\bf Sp}_{\E_*}}}
\newcommand{\ext}{\mathcal{E}}
\newcommand{\eno}{\mathbf{E}}
\newcommand{\epun}{{\mathbf{E}_*}}
\newcommand{\cel}{\operatorname{\text{\rm cell}}}
\newcommand{\rlp}{\operatorname{\text{\rm rlp}}}
\newcommand{\cof}{\operatorname{\text{\rm cof}}}
\newcommand{\calj}{{\mathcal{J}}}
\newcommand{\calk}{{\mathcal{K}}}
\newcommand{\cali}{{\mathcal{I}}}

\newcommand{\calit}{{\mathcal{I}}}

\newcommand{\calf}{{\mathcal{F}}}
\newcommand{\calh}{{\mathfrak{h}}}
\newcommand{\calhg}{{\mathcal{H}}}
\newcommand{\edelta}{\boldsymbol{\Delta}}
\newcommand{\calhr}{\tilde{\mathfrak{h}}}
\newcommand{\hre}{\tilde{h}}
\newcommand{\cofi}{{\mathbf{C}}}
\newcommand{\cw}{{\mathbf{CW}}}
\newcommand{\topo}{\operatorname{{\bf Top
}}}
\newcommand{\topopair}{\operatorname{{\bf Top
^{\mathrm{pair}}}}}
\newcommand{\cyl}{\operatorname{\text{\rm Cyl}\,}}
\newcommand{\Hom}{\operatorname{\text{\rm Hom}}}
\newcommand{\catss}{\operatorname{{\bf sset}}}

\newcommand{\cell}{\operatorname{{\rm cell}}}
\newcommand{\catab}{\operatorname{{\bf Ab}}}
\newcommand{\map}{\operatorname{{\rm map}}}
\newcommand{\Map}{\operatorname{{\rm Map}}}
\newcommand{\id}{\operatorname{{\rm id}}}
\newcommand{\eva}{\operatorname{{\rm ev}}}
\newcommand{\PP}{\mathbf{P}}
\newcommand{\timess}{\bar \times}
\newcommand{\Ho}{\operatorname{{\rm Ho}}}

\begin{document}

\title[Proper and exterior stable homotopy and cohomology theories]{A stable framework for proper and exterior homotopy and cohomology theories}

\author{Antonio Ceres, Jos\'e~M. Garc{\'\i}a-Calcines and Aniceto Murillo }

\begin{abstract}
We develop a stable framework for exterior homotopy theory and apply it to proper homotopy theory. Starting from a known simplicial model structure on the category of exterior spaces we then pass to the proper, cellular model category of retractive exterior spaces over $\ray$, and establish the associated stable homotopy theory of exterior spectra.

\smallskip

The resulting framework provides exterior analogues of several classical constructions in stable homotopy theory and sheds new light on proper homotopy theory, where the development of a satisfactory stable theory has been hindered by the lack of suitable categorical properties. As an application, we establish a Brown representability theorem for reduced exterior cohomology theories. This yields representability results for proper cohomology theories arising from the exterior setting, including cohomology theories with compact supports associated with classical generalized cohomology theories.
\end{abstract}

\maketitle

\tableofcontents

\section*{Introduction}
One of the central insights of modern algebraic topology is that many homotopical phenomena become more accessible after stabilization. The passage from spaces to spectra reveals structures that remain invisible at the unstable level. 
The aim of this paper is to develop such a stable framework in the setting of exterior spaces and consequently in proper homotopy theory. As in the classical setting, we also seek to represent exterior cohomology theories by spectra in the exterior context, including those arising as cohomological invariants in the proper setting.

Exterior homotopy theory constitutes a convenient environment for studying homotopy theory together with its behaviour at infinity. An exterior space is a topological space equipped with an externology, that is, a family of open subsets regarded as neighbourhoods of infinity. This additional structure makes it possible to encode asymptotic information while retaining much of the flexibility of ordinary homotopy theory.

The relationship between exterior spaces and proper topology is well known. Every proper space determines an exterior space via its cocompact externology, yielding a fully faithful embedding of the proper category $\PP$ into $\E$. Many constructions of proper homotopy theory then admit a more natural formulation in the exterior setting. Indeed, one of the main advantages of this approach is that, whereas the proper category lacks many of the functorial constructions that are essential in homotopy-theoretic arguments, the category $\E$ is complete and cocomplete. In particular, exterior homotopy groups, Brown--Grossman invariants and other algebraic tools can be developed in a natural and functorial way within $\E$, providing effective invariants for the study of non-compact spaces.

Despite these developments, a stable homotopy theory of exterior spaces has not been available. The reason is that stabilization requires considerably more than the existence of a homotopy category. In order to construct spectra, see \cite{ho1}, one needs a pointed, left proper and either combinatorial or cellular model category together with a pair of Quillen endofunctors playing the role, in the homotopy category, of suspension and loop functors.

The starting point of the present work is the analysis, under a modern approach, of the simplicial model structure on the category $\E$ introduced in \cite{calpiher}: let $\bsph^n$, $\bdisk^n$ and $\bcil^n$ denote the  Brown--Grossman  or exterior $n$-sphere, $n$-disk and $n$-cylinder introduced in
\cite{brown,por}, see Section~\ref{prelimi2} for the relevant definitions and notation. The {\em exterior model category} is the model structure on $\E$, cofibrantly generated by the  sets of maps
$
\cali=\{\bsph^{n-1}\to\bdisk^n\}_{n\ge0}$,
 $
\calj=\{\bdisk^n\to\bcil^n\}_{n\ge0}$,
and whose weak equivalences are the exterior maps inducing bijections on exterior homotopy classes of discrete rays and isomorphisms on all
exterior homotopy groups over any given discrete ray.

This model structure is cellular and right proper but it fails to be left proper and pointed. This is overcome by considering the retractive category $\epun$ under and over $\ray$ for which we prove (see Proposicion~\ref{oju} and Theorem~\ref{leftpro}):

\begin{theorem*} The exterior model structure induces in $\epun$ a simplicial cofibrantly generated, cellular, proper  model structure.
\end{theorem*}

Furthermore, we exlicitly describe how the tensor and cotensor constructions descend naturally to this pointed context. Most importantly, the resulting model structures do not alter the underlying notion of exterior homotopy. Indeed, the homotopy relation induced by the cylinder objects associated with the simplicial structure agrees with the classical notion of exterior homotopy under $\ray$. More precisely (see Proposition~\ref{rectificacion}):

\begin{proposition*} 
Let $X,Y\in \E_*$, with $X$ cofibrant and $Y$ fibrant. Then the natural map
$
 [X,Y]_{\E_*}\stackrel{\cong}{\longrightarrow} [X,Y]^\ray
$ is a bijection. 
\end{proposition*}
 
 As an immediate consequence the corresponding {\em exterior Whitehead theorem} is readily deduced (see Theorem \ref{white}):
 
  \begin{theorem*}
A map in $\epun$ between fibrant and cofibrant objects  is a weak equivalence if and only if it is an
exterior homotopy equivalence under $\ray$.
\end{theorem*}
 
Before proceeding to stabilization, see Section \ref{cofibrantes}, it is also important to understand the class of cofibrant objects associated with the model structures introduced above. We investigate the class of these objects in both $\E$ and $\E_*$ and show that the resulting notion of cofibrancy is sufficiently broad to encompass many of the geometric examples arising naturally in exterior and proper homotopy theory. In particular, locally finite CW-complexes fit naturally into the framework developed here.

 Now, the existence of the suspension--loop Quillen pair (Corollary~\ref{adjuncion})
$$ \xymatrix{ \E_*& \E_*\,, \ar@<0.75ex>[l]^(.50){\sus} \ar@<0.75ex>[l];[]^(.50){\lup}\\} $$
which is easily deduced from general properties of the pointed tensor and cotensor functors on $\E_*$, places us in the appropriate setting for stabilization. From this point onward, the general machinery of spectra in model categories becomes available. Following this  approach we define the category $\spece$ of exterior (sequential) spectra and equip them with the corresponding projective and stable model structures. 
As in classical stable homotopy theory, stable equivalences admit an intrinsic homotopical characterization. We show that they are detected by stable exterior homotopy groups  providing computable invariants for the resulting stable category. More precisely, given an exterior spectrum  we introduce, see Definition \ref{stableq},  the $k$-th {\em exterior stable homotopy group} of $X$,
$$
\pi_k^{s,e}(X)
=
\varinjlim_n
\pi_{k+n}^{e}(X_n),
$$
and prove, see Theorem \ref{stableequivstablegroups}:

\begin{theorem*} A map of exterior spectra is a stable equivalence if and only if it induces isomorphisms on exterior stable homotopy groups.
\end{theorem*}

The stable category $\spece$ of exterior spectra contains important analogues of the fundamental objects of classical stable homotopy theory. Moreover, unlike the classical case, it  
 is enriched by the existence of {\em function spectra} which are given by a bifunctor
$$
\calf\colon \E_*^{op}\times \spece\longrightarrow \spec
$$
which provides a stable representability mechanism for exterior homotopy classes. More precisely (see Theorem~\ref{homoto}):
\begin{theorem*}
For every cofibrant exterior space $X\in\cofi_*$ and every exterior $\Omega$-spectrum $Y\in\spece$,
$$
\pi_k^{st}\calf(X,Y)\cong [X,Y_k]_{\epun}\quad \text{for all $k\in\Z$.}
$$
\end{theorem*}

Furthermore, the stable category of exterior spectra is rich enough to support a Brown-type representability theorem for exterior cohomology theories, one of the principal outcomes of the theory. We begin by introducing reduced exterior cohomology theories as countable families of contravariant functors from the pointed exterior category $\E_*$ to the category of abelian groups satisfying the exterior analogues of the homotopy, exactness and additivity axioms (see Definition~\ref{redu}). Non-reduced theories are defined similarly. The stable machinery developed throughout the paper then provides all the ingredients needed to establish (see Theorem~\ref{repre} and Corollary~\ref{cororepre}):
\begin{theorem*}
For every reduced exterior cohomology theory $\calhr^*$ there exists an exterior $\Omega$-spectrum $E\in\spece$ such that, for every cofibrant object $X\in\E_*$ and every $n\ge 0$, there are natural isomorphisms
$$
\calhr^n(X)
\cong
[X,E_n]_{\E_*}
$$
\end{theorem*}

This representability theorem has important consequences in proper homotopy theory. Since proper spaces embed fully faithfully into the category of exterior spaces, every representable exterior cohomology theory determines a proper cohomology theory. By the theorem above, these theories are represented by exterior spectra, even though the representing objects generally lie outside the proper setting.

This phenomenon is particularly transparent for exterior cohomology theories arising from classical cohomology theories. Given a continuous generalized cohomology theory $h^*$, the Alexandroff exterior construction induces an exterior cohomology theory whose restriction to proper spaces coincides with the corresponding cohomology theory with compact supports (see Theorem \ref{clasicacoho} and Corollary \ref{compactsupport}). Applying the representability theorem, we conclude that these theories are represented by exterior spectra. In particular, classical cohomological invariants of non-compact spaces admit a natural interpretation within the stable exterior framework.

As an illustration, Theorem~\ref{alex} shows that, when $h$ is Alexander--Spanier cohomology with coefficients in an abelian group $G$, the associated exterior cohomology theory, which coincides with the compactly supported cohomology theory induced by $h$ in the proper setting, is represented  by an exterior Eilenberg--Mac Lane spectrum of type  $G^\infty$, a countable product of copies of $G$.

The detailed table of contents provides a largely self-explanatory guide to the organization of the paper and to the location of the various constructions and results discussed above.

\section{Preliminaries}\label{prelimi1}

In this section we review the necessary background on spectra in model categories and on the exterior category from a homotopical viewpoint. Throughout the paper, $\Hom_{\quic}$ denotes the set of morphisms in the category $\quic$.

\subsection{Stable structures in model categories}\label{prespec}

We recall the minimal assumptions, following \cite{ho1}, required to develop spectra in a general model category.
In what follows $\quic$ will denote   a left proper cellular model category  endowed with a Quillen pair of endofunctors  $\Sigma\dashv\Omega$. A (sequential) {\em spectrum} in $\quic$ is a family of objects $x=\{x_n\}_{n\ge 0}$ of $\quic$ together with {\em structure maps} $\sigma\colon\sus x_n\to x_{n+1}$, for $n\ge 0$, or equivalently, the  adjoint maps $\sigma^\adj\colon x_n\to \lup x_{n+1}$. A morphism of spectra $f\colon x\to y$ is a collection of maps $\{f_n\}_{n\ge 0}$, $f_n\colon x_n\to y_n$, compatible with the structure maps in the obvious sense. We denote by 
 $\spec(\quic)$ 
 the category of spectra in $\quic$ which  is bicomplete \cite[Lemma 1.3]{ho1}.

 The pair $\sus\dashv\lup$ prolongs  to adjoint endofunctors $\suse\dashv\lupe$ in $\spec(\quic)$ by
$$
 (\suse x)_n=\sus x_n, \qquad (\lupe x)_n=\lup x_{n}
 $$
 with structure maps
 $$
 \sus(\suse x)_n=\sus^2 x_n\stackrel{\sus\sigma}{\longrightarrow}\sus x_{n+1}=(\suse x)_{n+1},\qquad (\lupe x)_n=\lup x_n\stackrel{\lup\sigma^\adj}{\longrightarrow}\lup^2x_{n+1}=\lup(\lupe x)_{n+1}.
 $$

 For each $k\ge0$, the evaluation functor
\begin{equation}\label{eva}
 \eva_k\colon \spec(\quic)\longrightarrow \quic,\quad x\mapsto x_k,
\end{equation}
 has a left adjoint
\begin{equation}\label{susk}
 \sus^{\infty-k}\colon \quic\longrightarrow\spec(\quic),\quad (\sus^{\infty-k}a)_n=\begin{cases} \sus^{n-k}a&\text{if $n\ge k$},\\ \,\,\,0&\text{otherwise},\end{cases}
\end{equation}
being $0$ the initial object of $\quic$.

The {\em projective model structure} in $\spec(\quic)$, which is left proper and cellular by \cite[Thm.~A.9]{ho1}, is defined as follows: a map of spectra $f\colon x\to y$  is a projective fibration or a projective weak equivalence if $f$ is a levelwise fibration or a levelwise weak equivalence. As a result  \cite[Prop.~1.14]{ho1}, $f$ is  a projective cofibration if and only if $f_0$ and the induced maps $g_n\colon x_{n}{\sqcup_{\sus x_{n-1}}}\sus y_{n-1}\to y_n$, $n\ge 1$, are cofibrations.
 Moreover, if $ \mathcal I$ and $\mathcal J$ are the sets of generating cofibrations and trivial cofibrations of $\quic$, then the projective model structure in $\spec(\quic)$ is cofibrantly generated by the  sets $
 {\mathcal I}_{\sus}=\cup_{k\ge 0}\sus^{\infty-k}({\mathcal I})$ and ${\mathcal J}_{\sus}=\cup_{k\ge 0}\sus^{\infty-k}({\mathcal J})$, see \cite[Thm.~1.13]{ho1}.

 With respect to the projective model structure the adjoint pairs of functors $\suse\dashv\lupe$ and $\sus^{\infty-k}\dashv \eva_k$, $k\ge0$, are Quillen \cite[Prop.~1.15]{ho1}.

The {\em stable model structure} on $\spec(\quic)$ is constructed ad hoc to ensure that the adjunction $\suse\dashv\lupe$ is a Quillen equivalence.

For it,  recall 
 that given a  set of morphisms $S$ of a left proper cellular model category $\quic$ the {\em left Bousfield localization} of $\quic$ with respect to $S$ \cite[Thm.~4.1.1]{hirsch0} is a new model structure $\quic_S$ on $\quic$ together with a left Quillen functor $\quic\to\quic_S$ that is universal among left Quillen functors $F\colon \quic\to\quid$ for which $F(s)$ is a weak equivalence for every $s\in S$.
 The cofibrations and weak equivalences of $\quic_S$ are, respectively, the cofibrations of $\quic$ and the $S$-local equivalences. Fibrant objects in $\quic_S$ are the $S$-local fibrant objects of $\quic$, and a map between such  objects  is a weak equivalence in $\quic_S$ if and only if it is a weak equivalence in $\quic$. Left properness and cellularity are preserved by localization.

Provided $\quic$ is endowed with  a Quillen pair of endofunctors $\sus\dashv\lup$,  define the {\em stable model structure} in $\spec(\quic)$ as the Bousfield localization of the projective model structure with respect to  the following set $S$, see \cite[Def 3.3]{ho1}:  for any object $a\in \quic$ and for any $n\ge 0$ write   $\id_{\sus a}\colon\sus a\to \eva_{n+1}\sus^{\infty-n} a$ whose adjoint is denoted by $\zeta_n^a\colon 
 \sus^{\infty-(n+1)} \sus a\to \sus^{\infty-n} a$. Then,
\begin{equation}\label{equibous}
S=\{\zeta_n^{Qc}\}_{n\ge0,c}
\end{equation}
where $c$ runs through any domain or codomain of a set of generating cofibrations of $\quic$ and $Q$ denotes a functorial cofibrant replacement. 

 The fibrant objects in the stable structure are the {\em  $\lup$-spectra}, i.e.,  objects $x\in\spec(\quic)$ such that each $x_n$ is fibrant and, for all $n\ge 0$, the adjoint $x_n\to\lup x_{n+1}$ of the structure map is a weak equivalence  (see \cite[Thm.~3.4]{ho1}).

As intended, with respect to the stable structure, the 
 pair $\suse\dashv\lupe$ is a Quillen equivalence in $\spec(\quic)$ \cite[Thm.~3.9]{ho1}.

 \subsection{Exterior homotopy theory}\label{prelimi2}

Standard results from exterior homotopy theory can be found in \cite{calpiher} or in the overview provided in \cite[\S2]{gargarmu2}. Here, we provide only a concise account.

Throughout this paper, all spaces are assumed to be compactly generated and weakly Hausdorff. We denote their category by $\topo$, and by $\map(X,Y)$ the set of morphisms $\Hom_{\topo}(X,Y)$ from $X$ to $Y$, endowed with the compact-open topology. We write $\topo_*$ and $\map_*$ for the corresponding pointed category and mapping space, respectively.

 An \textit{exterior space} $(X,\ext)$ consists of a topological space $(X,\tau)$
together with a nonempty family of {\em exterior sets} $\ext\subseteq \tau $, called
\textit{externology}, which  is closed under finite intersections and,
whenever $U \supseteq E$, $E \in \ext$, $U\in \tau$, then $U\in \ext$. Roughly speaking, an exterior space
is a topological space with a neighbourhood system at infinity given by the externology. The complements of exterior sets are called exterior closed, or e-closed.  When $\ext=\tau$ we call it the {\em total} externology. A continuous map $f\colon (X,\ext ) \rightarrow (Z,\ext'
)$ is  \textit{exterior} if  $f^{-1}(E) \in \ext$, for all $E \in \ext'$. We denote by $\E$ the corresponding  category. For exterior spaces $X$ and $Z$ we often write $\map_\E(X,Z)$  for the morphism set  $\Hom_\E(X,Z)$ viewed as a subset of the mapping space $\map(X,Z)$. 

Given a collection $\mathcal F$ of open sets of a topological space $X$, the {\em externology generated by $\mathcal F$} is the coarsest externology containing it. Equivalently, it consists of all open sets that contain a finite intersection of elements of $\mathcal F$. For example, if   $Y\subseteq X$ is a subspace of an exterior space, the {\em relative externology}, defined as the coarsest externology on $Y$ that makes the inclusion map exterior, is generated by the sets  $E\cap Y$ as $E$ ranges over the exterior sets of $X$.
  
 The category $\E$ is complete and cocomplete \cite[Theorem 3.3]{calpiher}. Moreover, if we denote by $\mathbf{P}$  the category of spaces and proper continuous maps there is a full embedding
\begin{equation}\label{fullembed} \PP \hookrightarrow \E,\qquad X\mapsto X_{cc},\end{equation}
 where $X_{cc}$ denotes the space $X$ endowed with the \textit{cocompact externology},  formed by the
family of the complements of all closed compact subspaces. Throughout the text, whenever we write $X\in \PP$, we implicitly assume that $X$ is endowed with the cocompact externology.

A particularly interesting subcategory of $\PP$ is given by  $\PP_\infty$ consisting of non-compact, locally compact, $\sigma$-compact Hausdorff spaces. For any  $X\in\PP_\infty$, we may find an increasing sequence of compact subspaces $K_1\subset K_2\subset\dots \subset X$ such that $X=\cup_{m\ge 1} K_m$.

The product $X\times Z$ of two exterior spaces is  equipped with the {\em product externology} generated by the product  of exterior open sets of $X$ and $Z$. 

On the other hand, the {\em external product} $X\timess Y$ of an exterior space $X$ and a topological space $Y$ is the externology in $ X\times Y$ given by those open sets which, for each $y\in Y$, contain a product $F\times U$ of an exterior  set of $X$ and an open neighborhood of $y$. Note that, whenever $Y$ is compact, an open set of $X\times Y$ is exterior in $X\timess Y$ precisely when it contains $F\times Y$ for some exterior set $F\subset X$. That is $X\timess Y=X\times Y_{tr}$ in which $Y_{tr}$ denotes the space $Y$ endowed with the trivial externology $\{Y\}$. Moreover, for any topological space  $X$,
\begin{equation}\label{prodcom}
X_{cc}\timess Y=(X\times Y)_{cc}.
\end{equation} 

For any exterior space $X$,  the external product $X\timess I$ yields the appropriate homotopy notion: two maps $f,g\colon X\to Y$ in $\E$ are {\em (exterior or exteriorly)  homotopic}, and we write $f\simeq_e g$, if there is an exterior map $F\colon X\timess I\to Y$ such that $F_0=f$ and $F_1=g$. We denote by $[X,Y]$ the set of exterior homotopy classes of exterior maps from $X$ to $Y$. As $X_{cc}\timess I=(X\times I)_{cc}$ for any topological space $X$, the  embedding (\ref{fullembed}) again induces a full embedding
\begin{equation}\label{fullembed2} \PP/\simeq_{cc} \hookrightarrow \E/\simeq_e\end{equation}
from the {\em proper homotopy category}.

On the other hand, the {\em external mapping} $Z^Y$ of an exterior space $Z$ and a topological space $Y$ is defined by equipping $\map(Y,Z)$ with the externology generated by the sets  $(K,F)=\{f\colon Y\to Z,\,\,f(K)\subset F\}$ where $K$ ranges over the compact subspaces of $Y$ and $F$ over the exterior sets of $Z$. Note that if  $Y$ is compact, then $Z^Y$ is precisely $\map(Y_{tr},Z)$.

Furthermore, to obtain the appropriate exponential laws, given exterior spaces $X$ and $Z$, the morphism set   $\Hom_\E(X,Z)$  between the exterior spaces $X$ and $Z$ is equipped with a  refinement of  the compact-open topology. It is generated by the compact-open subbasis together with the sets $(L,F)$ as $F$ ranges over the exterior sets of $Z$ and $L$ ranges over the {\em $e$-compact} subsets of $X$, namely those $L$ for which $L\setminus E$ is compact for every exterior set $E\subset X$. We denote this topological space by $\map_\E(X,Z)$. With this notation, the following is \cite[Thm.~3.2]{calpiher}. 
In our setting the additional assumptions in loc.\ cit.\ can be omitted, 
since all spaces considered here are compactly generated weakly Hausdorff 
and the usual exponential law holds for this class.

\begin{theorem}\label{expo}
Let $X,Z\in\E$ be exterior spaces and let $Y$ be a topological space. Then there is a natural bijection
$$
\map_\E(X\timess Y,Z)\cong\map_\E(X,Z^Y).
$$
Moreover, if $X$ is endowed with the cocompact externology, there is a natural bijection
$$
\map_\E(X\timess Y,Z)\cong \map\bigl(Y,\map_\E(X,Z)\bigr).
$$\hfill$\square$
\end{theorem}

Given an exterior space $Z$, let $\E^Z=Z{\downarrow}\E$ denote the category of exterior spaces under $Z$. Its objects are pairs $(X,\alpha)$, or simply $X$ when $\alpha$ is understood, where $X\in \E$ and $\alpha \colon Z\to X$ is an exterior map. Morphisms in $\E^Z$ are given by commutative triangles under $Z$ whereas homotopy in $\E^Z$ is defined in the usual way: two maps   $f,g\colon X\to Y$ in $\E^Z$ are {\em homotopic relative to} $Z$ if there is an exterior homotopy $F\colon X\timess I\to Y$ such that, for each $t\in I$, the map $F_t\colon X\to Y$ is a morphism in $\E^Z$. We denote by $[X,Y]^Z$ the set of exterior homotopy classes relative to $Z$. 

As in the proper category $\PP$, the role of a continuous or discrete base point in an exterior space $X\in\E$ is played by a {\em ray}, namely an exterior map $\ray\to X$, or by a {\em discrete ray}, namely an exterior map $\N\to X$. Here, both $\ray=[0,\infty)$ and $\N$ are endowed with the cocompact externology.

\section{A simplicial, proper and cellular model structure on the exterior category}

We begin by reviewing  the simplicial model structure on $\E$ given  \cite[\S4]{calpiher}. 
As for the simplicial character, the functor of {\em function complexes} is given by
$$
 \E^{op}\times \E\to\catss,\quad (X,Y)\mapsto \Map_\E(X,Y),
$$
where
$$
\Map_\E(X,Y)_n=\Hom_\E(X\timess|\Delta[n]|,Y),
$$
with $|\cdot|$ denoting geometrical realization and  $\Delta[n]$ the standard $n$-simplex. The {\em composition map} 
$$
\Map_\E(X,Y)\times \Map_\E(Y,Z)\longrightarrow \Map_\E(X,Z)
$$
sends a pair $f,g$ of $n$-simplices to the map
$$
X\timess|\Delta[n]|\stackrel{\id_X\timess\Delta}{\longrightarrow}X\timess(|\Delta[n]|\times
|\Delta[n]|) \stackrel{f\timess\id_{|\Delta[n]|}}{\longrightarrow}Y\timess|\Delta[n]|\stackrel{g}{\longrightarrow} Z.
$$

On the other hand, the  tensors and cotensors  are given, respectively, by the functors
$$
\E\times \catss^f\to\E,\quad (X,K)\mapsto X\otimes K=X\timess |K|,
$$
and
$$
\E\times {\catss_*^f}^{op}\to\E,\quad (X,K)\mapsto X^K=X^{|K|},
$$
where $\catss^f$ denotes the full subcategory of  $\catss$ consisting of finite simplicial sets. 
In particular, we have natural isomorphisms
$$\Map_\E(X,Y)_0\cong \Hom_\E(X,Y)
$$
and
$$
\Map_{\catss}\bigl(K,\Map_\E(X,Y)\bigr)\cong \Map_\E(X\timess |K|,Y)\cong\Map_\E(X,Y^{|K|}).
$$
Here, $\Map_{\catss}$ denotes the usual simplicial mapping space.

The model structure on $\E$ introduced in \cite{calpiher}, which we now recall, is compatible with this simplicial structure and closely resembles Quillen's original simplicial model structure on topological spaces, whose weak equivalences are the weak homotopy equivalences and whose generating cofibrations and generating trivial cofibrations are
$$
I=\{ S^{n-1}\to D^n\}_{n\ge0}
\quad \text{and} \quad
J=\{ D^n\to D^n\times I\}_{n\ge0},
$$
respectively.

\begin{definition}\label{pesfera}  Given $n\ge 0$, the {\em $n$th Brown--Grossman sphere}, or simply the {\em exterior sphere}, is the exterior space $(\bsph^n, \xi)\in\E^\N$  where 
$$
\bsph^n=\N\timess S^n
$$
and $\xi(m)=(m,x_0)$, being $x_0$ the chosen base point of the $n$-sphere.

Analogously, the $n$th {\em Brown--Grossman disk}, or simply {\em exterior disk}, and {\em cylinder} are defined by
$$
\bdisk^n=\N\timess D^n,\qquad \bcil^n=\N\timess(D^n\times I).
$$
As before, these exterior spaces are canonically endowed with discrete rays and regarded as objects of $\E^\N$. 
 
 Finally, we denote by
$$
\mathcal{I}=\{\bsph^{n-1}\to \bdisk^n\}_{n\ge 0}
\quad\text{and}\quad
\mathcal{J}=\{\bdisk^{n}\to \bcil^n\}_{n\ge 0}
$$
the images of $I$ and $J$ under the functor
$\N\timess-\colon\topo_*\to\E^\N$.
\end{definition}

\begin{definition}\label{pgrupo}
Let $(X,\alpha)\in\E^\N$. The {\em $n$th Brown--Grossman homotopy group}, or simply the {\em exterior homotopy group}, of $(X,\alpha)$ is
$$
\pi_n^e(X,\alpha)=[\bsph^n,X]^\N.
$$
It is endowed with the unique group structure for which the bijection
$$
[\bsph^n,X]^\N\cong\pi_n(\map_\E(\mathbb N,X),\alpha),
$$
induced by the natural identification
$$
\Hom_{\E^\N}\bigl((\bsph^n,\xi),(X,\alpha)\bigr)
\cong
\map_*\bigl((S^n,x_0),(\map_\E(\mathbb N,X),\alpha)\bigr)
$$
provided by Theorem \ref{expo}, is a group isomorphism. A clear and concise survey of these groups can be found in [20, §3].
\end{definition}

\begin{definition}\label{we}
A map $f\colon X\to Y$ in $\E$ is an  {\em exterior weak equivalence} if it induces a bijection on discrrete rays and, for each of them, an isomorphism on the corresponding Brown--Grossman homotopy groups. More precisely, 
$$
f_*\colon[\N,X]\stackrel{\cong}{\longrightarrow} [\N,Y]
$$
is a bijection and, for each discrete ray $\alpha\colon\N\to X$, the induced map 
$$
\pi_n^e(f,\alpha)\colon \pi_n^e(X,\alpha)\stackrel{\cong}{\longrightarrow} \pi_n^e(Y,f\alpha)
$$
is an isomorphism for all $n\ge 1$.
\end{definition}
 The following collects some useful observations.

\begin{rem}\label{homoequi}
(i) For any discrete ray $\alpha\colon\N\to X$, the set
$$
\pi_0^e(X,\alpha)=[(\bsph^0,\xi),(X,\alpha)]^\N
$$
is naturally identified with $[\N,X]$. Thus, the above definition is the direct analogue of the classical definition of weak homotopy equivalence, which requires a bijection on $\pi_0$ together with isomorphisms on the higher homotopy groups.

\smallskip

(ii) Observe  that any exterior map between spaces without discrete rays is automatically an exterior weak equivalence. In particular, any map
$
f\colon X\to Y
$
between exterior spaces endowed with the total externology is an exterior weak equivalence, since such spaces admit no discrete rays. For example, any map between compact spaces endowed with the cocompact externology is an exterior weak equivalence.

\smallskip

(iii) As in the classical case, if $f,g\colon X\to Y$ are exterior homotopic maps and $\alpha\colon\N\to X$ is a discrete ray, then the induced homomorphisms on exterior homotopy groups
$$
\pi_*^e(f,\alpha)\colon \pi_*^e(X,\alpha)\to\pi_*^e(Y,f\alpha),
\qquad
\pi_*^e(g,\alpha)\colon \pi_*^e(X,\alpha)\to\pi_*^e(Y,g\alpha)
$$
agree up to the canonical isomorphism
$
\pi_*^e(Y,f\alpha)\cong \pi_*^e(Y,g\alpha)
$
induced by the homotopy. Therefore, every exterior homotopy equivalence is an exterior weak equivalence.
\end{rem}

\begin{theorem}\label{modelosimp}
The category $\,\E$ admits a simplicial, cofibrantly generated model structure with exterior weak equivalences as weak equivalences, generating cofibrations $\mathcal{I}$, and generating trivial cofibrations $\mathcal{J}$.
\end{theorem}

This theorem is essentially \cite[Thm.~4.1]{calpiher}. Nevertheless, we include a short proof, reformulated in the modern language of Hirschhorn's recognition theorem for cofibrantly generated model categories, pointing out how some of the assertions in \emph{loc.\ cit.} are used in verifying some of the hypotheses of the recognition theorem. We also provide a shorter, perhaps more conceptual, proof of the simplicial nature of the resulting model structure, which we henceforth refer to as the \emph{exterior model structure}.
 
 \begin{proof} (i) {\em There is a model structure as stated.} 
 
 \smallskip
 
 As usual, for $A=\mathcal{I},\mathcal{J}$, we denote by $\cel(A)$ the class of relative $A$-cell complexes obtained as transfinite composition of pushouts of coproducts of maps of $A$. We denote by $\cof(A)$ the class of retracts of morphisms of $\cel(A)$, and by  $\rlp(A)$ the class of morphisms having the rlp with respect to every morphism of $A$. Finally, we denote by $W$ the class of exterior weak equivalences.
  
 Recall that the recognition theorem for cofibrantly generated model categories  \cite[Thm.~11.3.1]{hirsch0} proves the existence of a model structure as asserted, once we verify the following: $W$ satisfies two-out-of-three property; both $\mathcal{I}$ and $\mathcal{J}$ permit the small object argument over a certain ordinal, $\N$ in our case; $\cof(\mathcal{J})\subset \cof(\mathcal{I})\cap W$; and $\rlp(\mathcal{I})=\rlp(\mathcal{J})\cap W$.
  
  The first condition trivially holds. Furthermore, it is clear that $W$ is also closed for retracts.
  
The fact that the domains of the morphisms in $\mathcal{I}$ (resp. $\mathcal{J}$) are small relative to $\mathcal{I}$ (resp. $\mathcal{J}$) with respect to $\omega$, and therefore with respect to any larger cardinal, is a particular instance of \cite[Prop.~4.2]{calpiher}.

  To prove the inclusion $\cof(\mathcal{J})\subset \cof(\mathcal{I})\cap W$ we first check that every map $\bdisk^n\to \bcil^n$ in $\calj$ is an exterior weak equivalence lying in $\cel (I)$.  Indeed, each of these maps admits an exterior retraction $\bcil^n\to \bdisk^n$, induced by the retraction $D^n\times I\to D^n$, which makes $\bdisk^n$  an exterior strong deformation retract of $\bcil^n$. In particular, by Remark \ref{homoequi}(iii),  $\bdisk^n\to \bcil^n$ is an exterior weak equivalence.

  On the other hand, the usual product CW structure on $D^n\times I$ relative to $D^n\times{0}$, which expresses the inclusion
$
D^n\cong D^n\times{0}\to D^n\times I
$
as a relative $I$-cell complex,
translates to the exterior setting showing that $\bdisk^n\to \bcil^n$ is a relative $\mathcal{I}$-cell complex.

Next, observe that for any exterior pushout of the form
$$
\xymatrix{
\bdisk^n\ar[r] \ar[d]& X \ar[d]
 \\
\bcil^n\ar[r]&
X\cup_{\bdisk^n}\bcil^n,}
$$
the exterior space $X$ is again an exterior strong deformation retract of   $X\cup_{\bdisk^n}\bcil^n$. It follows that every relative $\mathcal{J}$-cell lies in $\cof(\mathcal{I})\cap W$. Finally, since $W$ and $\cof(\mathcal{I})$ are closed under retracts, we conclude that $\cof(\mathcal{J})\subset \cof(\mathcal{I})\cap W$.

Finally, the equality $\rlp(\mathcal{I})=\rlp(\mathcal{J})\cap W$ is precisely \cite[Prop.~4.1]{calpiher}.

 \smallskip
 
 (ii) {\em The exterior model structure is simplicial.}

\smallskip

By \cite[Prop.~9.3.7(3)]{hirsch0}, it suffices to verify the pushout-product axiom. Recall that if $u\colon A\to B$ is a morphism in $\E$ and $v\colon K\to L$ is a morphism in $\catss^f$, their pushout-product is the exterior map
$$
u\,\square\, v\colon
(B\otimes K)\cup_{A\otimes K}(A\otimes L)
\longrightarrow
B\otimes L.
$$
The pushout-product axiom states that if $u$ is a cofibration in $\E$ and $v$ is a cofibration in $\catss$, then $u\,\square\, v$ is a cofibration in $\E$, which is trivial whenever either $u$ or $v$ is trivial.
However, since the model structure on $\E$ is cofibrantly generated, it suffices to check the pushout-product axiom on the generating cofibrations and generating trivial cofibrations.

Consider the functor
$$
F=\N\timess-\colon\topo_*\longrightarrow\E^\N,
$$
which, by Theorem~\ref{expo}, is a left adjoint. Hence it preserves pushouts and transfinite compositions, so that
$$
F(\cof(I))\subseteq\cof(\mathcal I),
\qquad
F(\cof(J))\subseteq\cof(\mathcal J).
$$
Moreover, for every map $i$ in $\topo_*$ and every simplicial map $v$, there is a natural isomorphism
$$
F(i)\square v\cong F(i\square v),
$$
since
$
(\N\timess A)\otimes K
\cong
\N\timess(A\times|K|)
$
and again, $F$ preserves pushouts.

Let $i\in I$, $j\in J$, and let
$
u\colon\partial\Delta[m]\to\Delta[m]$,
$
v\colon\Lambda^r[m]\to\Delta[m]
$
be a generating cofibration and a generating trivial cofibration of $\catss$, respectively. Since Quillen's model structure on $\topo_*$ is simplicial,
$$
i\square u\in\cof(I),\qquad
i\square v\in\cof(J),\qquad
j\square u\in\cof(J).
$$
Applying $F$ gives
$$
F(i)\square u\in\cof(\mathcal I),\qquad
F(i)\square v\in\cof(\mathcal J),\qquad
F(j)\square u\in\cof(\mathcal J),
$$
which is precisely the pushout-product axiom.
\end{proof}

\begin{rem}\label{observa}
Note that, as for any cofibrantly generated model category in which the
generating cofibrations are monomorphisms, the notions of compactness
and smallness relative to $\cell(\calit)$ in $\E$ are equivalent, see \cite[Prop~10.8.7]{hirsch0}.
Indeed, let
$$
X_0 \longrightarrow X_1 \longrightarrow \cdots \longrightarrow X_\beta
\longrightarrow \cdots \longrightarrow X_\lambda=\varinjlim_{\beta<\lambda} X_\beta
$$
be a transfinite composition in $\cell(\calit)$ indexed by an ordinal
$\lambda$. Then the canonical map
\begin{equation}\label{small}
\varinjlim_{\beta<\lambda} \Hom_\E(K,X_\beta)\longrightarrow
\Hom_\E(K,X_\lambda)
\end{equation}
is always injective. To see this, consider two elements of
$\varinjlim_{\beta<\lambda} \Hom_\E(K,X_\beta)$ represented by
$f\colon K\to X_\gamma$ and $g\colon K\to X_\gamma$ for some
$\gamma<\lambda$, and assume that the induced maps $i_\gamma f,i_\gamma g$
in $\Hom_\E(K,X_\lambda)$ coincide, where $i_\gamma\colon X_\gamma\to
X_\lambda$ is the canonical map. Since each map in $\cali$
is a monomorphism in $\E$, and monomorphisms are preserved under
coproducts, pushouts, and transfinite compositions, it follows that
$i_\gamma$ is a monomorphism, and hence $f=g$.
Finally, recall that, by definition, $K$ is $\lambda$-small if
(\ref{small}) is an isomorphism, which by the above is equivalent to
being surjective. This is precisely the definition of $K$ being
$\lambda$-compact.
\end{rem}

\begin{proposition}\label{celular} The exterior model category $\E$ is cellular.
\end{proposition}
\begin{proof}
In view of Remark~\ref{observa} and \cite[Def.~12.1.1]{hirsch0}, it remains to verify that the domains and codomains of the generating cofibrations are compact relative to $\cell(\calit)$ for some cardinal, and that cofibrations are effective monomorphisms.

The first condition follows from \cite[Prop.~4.2]{calpiher}, which implies that the domains and codomains of the maps in $\calit$ are $\omega$-small relative to $\cell(\calit)$. Hence they are compact relative to $\cell(\calit)$.

For the second, observe that each map in $\calit$ is an inclusion, and
that inclusions are effective monomorphisms in $\E$. Since cofibrations
are retracts of relative $\calit$-cell complexes, they are also
inclusions and thus effective monomorphisms.
\end{proof}

\begin{proposition}\label{celpro} 
The exterior model category $\E$ is right proper.
\end{proposition}

\begin{proof}
In view of \cite[Cor.~13.1.13]{hirsch0} right properness follows from the fact that every object $X\in \E$ is fibrant (cf. Remark \ref{fibrant}). Indeed, the terminal object of $\E$ is a point with the trivial externology. Since every generating trivial cofibration  is an exterior strong deformation retract, the unique map $X \to *$ has the right lifting property with respect to all of them.
\end{proof}

\begin{rem}\label{noleft} 
 However, the exterior model category $\E$ fails to be left proper as the following example shows.

Consider the pushout
$$
\xymatrix{
\varnothing \ar[r] \ar@{^{(}->}[d] &
\ast \ar[d]
\\
\N \ar[r]^-j &
\N\sqcup \ast,
}
$$
where $\ast$ denotes the one-point space endowed with the total
externology.

The map
$
\varnothing\longrightarrow *
$
is an exterior weak equivalence (see (ii) of Remark \ref{homoequi}) while the left vertical map is the generating cofibration $\bsph^{-1}\to\bdisk^0$. However, 
$j$ 
is not an exterior weak equivalence since $j_*\colon [\N,\N]\to[\N,\N\sqcup *]$ is not a bijection. Indeed, let
$
\alpha\colon\N\to\N\sqcup\ast
$
be any discrete ray whose image contains $*$ and write $\alpha(m)=*$. Then, 
$\alpha$ cannot be exteriorly
homotopic to any map whose image is contained in $\N$. Indeed, if
$
H\colon\N\timess I\to\N\sqcup\ast
$
were such a homotopy, its restriction to the connected subset
$\{m\}\times I$ would be constant with value $*$  contradicting the fact that
$
H(m,1)\in\N
$.
Hence $[\alpha]$ does not belong to the image of $j_*$.
\end{rem}

\section{The pointed exterior category {\bf E}$_*$}

According to \S\ref{prespec}, in order to obtain a satisfactory stabilization, it is necessary to work in a pointed category, which is not the case for $\E$. In fact, the initial object of this category is the empty set with the only topology and externology, whereas as observed above, the terminal object is a point with the trivial externology. To overcome this issue, we consider the category $\E_*$ of \emph{retractive exterior spaces over $\ray$}. An object is a diagram
$$
\ray\stackrel{s}{\longrightarrow}X\stackrel{r}{\longrightarrow}\ray
$$
such that $r\circ s=\id_{\ray}$, and morphisms are the obvious commutative diagrams. 
Equivalently, $\E_*$ is naturally identified with the slice category over $\id_{\ray}$ of the coslice category $\ray\downarrow\E$.
This category is not unduly restrictive as we will see in subsequent sections. 

\begin{proposition}\label{oju}
The category $\E_*$ inherits from $\E$ a cellular,  right proper, cofibrantly generated model structure in which a map is a fibration, cofibration, or weak equivalence if and only if its image under the forgetful functor $\epun \to \eno$ is a fibration, cofibration, or weak equivalence, respectively.
\end{proposition}

\begin{proof}
By \cite[Thms.~1.19,~2.20]{hirsch1}  any slice or coslice category of a model category inherits a model structure as stated. Moreover, by \cite[Thms.~1.23,~1.24,~2.23,~2.24]{hirsch1}, cellularity and right properness are also preserved under the slice and coslice constructions.
\end{proof}

\begin{rem}\label{nuevacof}
The set of generating cofibrations and trivial cofibrations of the induced model structure on $\E_*$ can be explicitly given as follows:  for each map $i\colon A\to B$ in $\mathcal I$
and each exterior map $q\colon B\to\ray$, consider the retractive objects
$$
\left(\ray \hookrightarrow A\sqcup\ray \xrightarrow{\,q\circ i\,\sqcup\,\id\,}
\ray\right)
\quad\text{and}\quad
\left(\ray \hookrightarrow B\sqcup\ray \xrightarrow{\,q\,\sqcup\,\id\,}
\ray\right),
$$
where the section is the inclusion of the $\ray$-summand. The induced map
$$
A\sqcup\ray\longrightarrow B\sqcup\ray
$$
is then a generating cofibration of the global (resp. exterior) model
structure on $\E_*$. The generating trivial cofibrations are defined
analogously.
\end{rem}

\begin{rem}\label{fibrant}
In contrast to $\E$, where every object is fibrant (see the proof of Theorem~\ref{celpro}), not every object of $\E_*$ is fibrant: let $X=\ray$ with section
$i(t)=t+2$ and retraction
$$
p(t)=
\begin{cases}
t, & 0\le t\le1,\\
2-t, & 1\le t\le2,\\
t-2, & t\ge2.
\end{cases}
$$
Consider the commutative square
$$
\xymatrix{
\bdisk^0 \ar[r]^{a} \ar[d] &
X \ar[d]^{p} \\
\bcil^0 \ar[r]_{g} &
\ray,
}
$$
where $a(n)=n$, $g(n,s)=p(n)$ for $n\neq1$, and
$g(1,s)=1+s$. If this square admitted a lifting $h$, writing
$\lambda(s)=h(1,s)$ we would have a continuous map
$\lambda\colon I\to\ray$ such that
$$
\lambda(0)=1
\quad\text{and}\quad
p(\lambda(s))=1+s.
$$
This is impossible since, for $s>0$ sufficiently small, continuity
implies that $\lambda(s)$ is close to $1$, whereas
$p(t)\le1$ for all $t$ sufficiently close to $1$. Hence $p$ is not an
exterior fibration, and therefore the corresponding object of
$\E_*$ is not fibrant.
\end{rem}

\begin{definition}\label{tensorycotensor}
Let
$
X\in \E_*
$
and  $K\in \catss^f$.
\begin{enumerate}

\item The \emph{tensor} of $X$ by $K$, denoted by $X{\otimes_\ray }K\in\epun$, is defined as the pushout in $\E$ of the maps $
\ray \leftarrow \ray\otimes K \to X\otimes K$  induced by the ray of $X$ and the projection $\ray\timess| K| \to \ray$. That is,
$$
\xymatrix{
\ray\otimes K \ar[r]^{s_X\otimes \id_K} \ar[d] &
X\otimes K \ar[d] \\
\ray \ar[r] &
X{\otimes_\ray} K,
}
$$
with the evident ray and projection over $\ray$.

\item The \emph{cotensor} of $X$ by $K$, denoted by $X^K_\ray\in\epun$,  is defined as the pullback in $\E$ of the maps in $\ray\to \ray^K\leftarrow X^K$ induced again by the ray of $X$ and the constant map at any point of $\ray$. That is,
$$
\xymatrix{
X^K_\ray \ar[r] \ar[d] &
X^K \ar[d] \\
\ray \ar[r] &
\ray^K,
}
$$
with the natural ray and projection over $\ray$.

\item Finally, for $X,Y\in \epun$, the \emph{simplicial mapping space} $\Map_{\epun}(X,Y)$ is the simplicial set defined by   
$$
\Map_{\E_*}(X,Y)_n=\Hom_{\epun}(X\otimes_\ray \Delta[n],\,Y),
$$
which is pointed by the zero morphism
$$
X\longrightarrow \ray \longrightarrow Y.
$$
Equivalently, it is the simplicial subset of $\Map_\E(X,Y)$ consisting of those simplices
$$
f\colon X\otimes \Delta[n]\longrightarrow Y
$$
in $\E$ which are exterior maps over and under $\ray$.
\end{enumerate}
\end{definition}

\begin{theorem}\label{sinpuntear}
With the tensor, cotensor and simplicial mapping spaces defined above,
both the exterior and the global model structures on $\E_*$ are simplicial
model categories.
\end{theorem}
\begin{proof}
This follows from the fact that the slice and coslice categories of a simplicial model category, and hence the associated retractive category, inherit a simplicial model structure with respect to the induced tensor and cotensor, see \cite[II.2, Prop.~6]{qui}.
\end{proof}

In particular:

\begin{corollary}\label{coropunsin}
(i) For any $X,Y \in \E_*$ and any $K \in \catss^f$, there are natural
isomorphisms of simplicial sets
$$
\Map_{\E_*}(X\otimes_\ray K, Y)
\cong
\Map_{\catss}\bigl(K,\Map_{\E_*}(X,Y)\bigr)
\cong
\Map_{\E_*}(X, Y^{K}_\ray).
$$

(ii) For any $X \in \E_*$ and any $K,L \in \catss^f$, there are natural
isomorphisms in $\E_*$
$$
X \otimes_\ray \Delta[0] \cong X \cong X^{\Delta[0]}_\ray,
\quad
X \otimes_\ray (K \times L) \cong (X \otimes_\ray K)\otimes_\ray L,
\quad
X^{K \times L}_\ray \cong (X^K_\ray)^L_\ray.
$$

(iii) For any $K \in \catss^f$, the following adjunction forms a Quillen pair,
$$
{-}\otimes_\ray K \dashv (-)^K_\ray
$$
\end{corollary}

\begin{definition}\label{cilpath}
(i) The \emph{cylinder of $X \in \E_*$} is defined as
$$
\cyl(X)= X \otimes_\ray \Delta[1].
$$
Since $\partial\Delta[1]\cong \Delta[0]\sqcup \Delta[0]$ in $\catss^f$, there is a natural isomorphism
$$
X \otimes_\ray \partial\Delta[1] \cong X \sqcup_{\ray} X,
$$
where the coproduct is taken in the retractive category $\E_*$. 
Thus, the map
$$
X \otimes_\ray \partial\Delta[1] \longrightarrow X \otimes_\ray \Delta[1],
$$
induced by the inclusion $\partial\Delta[1]\hookrightarrow \Delta[1]$,
exhibits $\cyl(X)$ as a functorial cylinder object for $X$, and hence as a good cylinder object for cofibrant objects in $\E_*$.

\smallskip

(ii) Dually, the \emph{path space} of $X$ is defined as
$$
\mathrm{Path}(X)= X^{\Delta[1]}_\ray.
$$
Since $\partial\Delta[1]\cong \Delta[0]\sqcup \Delta[0]$ in $\catss^f$, there is a natural isomorphism
$$
X^{\partial\Delta[1]}_\ray \cong X \times_{\ray} X,
$$
where the pullback is taken in $\E_*$. It follows that the map
$$
X^{\Delta[1]}_\ray \longrightarrow X^{\partial\Delta[1]}_\ray,
$$
induced by the inclusion $\partial\Delta[1]\hookrightarrow \Delta[1]$, exhibits $\mathrm{Path}(X)$ as a functorial  path object for $X$, and hence as a good path object for any fibrant object $X\in\epun$.
\end{definition}

\begin{proposition}\label{rectificacion}
Let $X,Y\in \E_*$, with $X$ cofibrant and $Y$ fibrant. Then the natural map
\begin{equation}\label{map}
 [X,Y]_{\E_*}\stackrel{\cong}{\longrightarrow} [X,Y]^\ray
\end{equation}
is a bijection. 
\end{proposition}

\begin{proof}
Note first that (\ref{map}) is well defined. Indeed, if $f,g\colon X\to Y$
are homotopic in $\E_*$ then, see Definition \ref{cilpath}, there exists a morphism
$$
F\colon X\otimes_\ray \Delta[1]\longrightarrow Y
$$
in $\E_*$ restricting to the two copies of $X$ inside
$
X\otimes_\ray \partial\Delta[1]\cong X\sqcup_\ray X
$.
Composing $F$ with the map
$
q\colon X\timess I=X\otimes \Delta[1]\to X\otimes_\ray \Delta[1]
$
yields an exterior homotopy
$$
H=F\circ q\colon X\timess I\longrightarrow Y
$$
from $f$ to $g$ under $\ray$.

To prove surjectivity of (\ref{map}), let $f\colon X\to Y$ be an exterior map under $\ray$ and denote by  $i$ the sections of $X$ and $Y$ from the ray, and by $p$ their projections onto it. 

Since  $pf=p$ on $i(\ray)$ the map
$$
K(x,t)=(1-t)\,pf(x)+t\,p(x)
$$
defines an exterior homotopy from $pf$ to $p$, relative to $i(\ray)$.

Consider the exterior inclusion
$$
j\colon
X\timess\{0\}\sqcup_{i(\ray)\timess\{0\}}i(\ray)\timess I
\longrightarrow X\timess I.
$$
Under the natural identifications
$$
X\otimes\Delta[0]\cong X,
\qquad
X\otimes\Delta[1]\cong X\timess I,
$$
this map is the pushout-product of the cofibration
$
i\colon\ray\to X
$
with the trivial cofibration
$
\Delta[0]\hookrightarrow\Delta[1]
$
of simplicial sets. Hence, since $X$ is cofibrant, $j$ is a trivial cofibration. Moreover, since $Y$ is fibrant,  $p\colon Y\to\ray$ is a fibration and  the commutative square
$$
\xymatrix{
X\timess\{0\}\sqcup_{i(\ray)\timess\{0\}}
i(\ray)\timess I
\ar[r]
\ar@{^{(}->}[d]_j
&Y\ar[d]^{p}\\
X\timess I
\ar[r]_-K
\ar@{-->}[ur]
&\ray
}
$$
admits a lifting, where the upper map is given by $f$ on $X\times\{0\}$ and by $i$ on $\ray\times I$.
The restriction of the lifting to $X\times\{1\}$ defines a map
$
\widetilde f\colon X\to Y
$
in $\epun$  and the lifting itself provides an exterior homotopy under $\ray$ from $f$ to $\widetilde f$. Hence every element of $[X,Y]^\ray$ is represented by a morphism in $\E_*$, proving surjectivity.

To check injectivity let $f,g\colon X\to Y$ be morphisms in $\E_*$ which are
exteriorly homotopic under $\ray$ via an exterior homotopy
$
H\colon X\timess I\longrightarrow Y
$
under $\ray$.

As before, we denote by $i$ and $p$ the sections and projection over the ray of $X$ and $Y$.
Set
$$
F=p \circ H\colon X\timess I\longrightarrow \ray\quad\text{and}\quad  
G=p \circ \pi\colon X\timess I\longrightarrow \ray,
$$
where $\pi\colon X\timess I\to X$ is the projection, and note that both homotopies coincide on
$$
A=X\times\{0,1\}\,\cup\, \ray\times I\subset X\timess I.
$$
Then, the map
$$
K\colon (X\timess I)\timess I\longrightarrow \ray,\qquad K\bigl((x,t),s\bigr)=(1-s)\,F(x,t)+s\,G(x,t)
$$
defines an exterior homotopy
from $F$ to $G$, relative to $A$.

Consider the subspace
$$
B=(X\timess I)\times\{0\}\,\cup\, A\times I
\subset (X\timess I)\timess I
$$
and observe that the exterior map
$$
h\colon B\longrightarrow Y,
$$
which extends $H$ on $(X \times I)\times\{0\}$, and  is defined on $A\times I$ by 
$$
h(x,0,s)=f(x),\quad 
h(x,1,s)=g(x),\quad 
h(r,t,s)=i(r),
$$
fits in the following commutative square of $\E$
$$
\xymatrix{
B \ar[r]^h \ar@{^{(}->}[d]_j & Y \ar[d]^{p}\\
(X\timess I)\timess I \ar[r]_-K \ar@{-->}[ur]^{\widetilde K} & \ray.
}
$$
We verify the existence of the lift $\widetilde K$. Since $Y$ is fibrant, the projection
$
p\colon Y\to\ray
$
is a fibration in $\E$. On the other hand, consider the trivial cofibration of simplicial sets
$$
L=(\Delta[1]\times{0})\cup
(\partial\Delta[1]\times\Delta[1])
\hookrightarrow
\Delta[1]\times\Delta[1].
$$
 Under the natural identifications
$|\Delta[1]|=I$ and $|\partial\Delta[1]|=\{0,1\}$, the inclusion
$$
j\colon B\longrightarrow (X\timess I)\timess I
$$
is then the pushout-product of the cofibration $i\colon \ray\to X$ and 
$L\hookrightarrow\Delta[1]\times\Delta[1]$.
By the pushout-product axiom, $j$ is a trivial cofibration, and the lift $\widetilde K$ exists.
Furthermore, it defines a map
$$
\widetilde H\colon X\timess I\longrightarrow  Y,
\quad \widetilde H(x,t)=\widetilde K((x,t),1).
$$
such that
$$
\widetilde H(x,0)=f(x),\quad \widetilde H(x,1)=g(x),
\quad \widetilde H(r,t)=i(r)\quad\text{and}\quad p\circ \widetilde H = G = p\circ \pi.
$$
Thus $\widetilde H$ is an exterior homotopy from $f$ to $g$ both under and
over $\ray$ and consequently, it induces a homotopy
$
X\otimes_\ray \Delta[1]\longrightarrow Y
$
in $\E_*$ between $f$ and $g$. 
\end{proof}
 
As an immediate consequence we obtain:
 
 \begin{theorem}[Whitehead theorem in $\E_*$]\label{white}
A map in $\epun$ between cofibrant-fibrant objects  is a weak equivalence if and only if it is an
exterior homotopy equivalence under $\ray$.
\end{theorem}

\begin{proof}
A map in $\epun$ between between cofibrant-fibrant objects is a weak equivalence if and only
if it is a homotopy equivalence in $\epun$.

Assume first that $f$ is a homotopy equivalence in $\E_*$, with homotopy inverse $g$. Since $f$ and $g$ are morphisms in $\E_*$, Proposition~\ref{rectificacion} implies that the corresponding homotopies in $\E_*$ determine exterior homotopies under $\ray$. Hence $g$ is an exterior homotopy inverse of $f$ under $\ray$.

Conversely, let $f$ be an exterior homotopy equivalence under $\ray$, with exterior homotopy inverse $g$. By Proposition~\ref{rectificacion}, $g$ is exteriorly homotopic under $\ray$ to a morphism
$
g'\colon Y\to X
$
in $\E_*$. The same proposition then implies that $g'$ is a homotopy inverse of $f$ in $\epun$.
\end{proof}

We conclude this section by establishing a property of $\epun$ that is essential for its stabilization and is not inherited directly from $\E$, as shown in Remark~\ref{noleft}. The presence of a discrete ray in every object of $\epun$ rules out the type of counterexample exhibited in in {\em loc.\ cit.}

\begin{theorem}\label{leftpro} The model category $\epun$ is left proper. 
\end{theorem}
\begin{proof}
Consider a pushout in $\epun$
$$
\xymatrix{
A \ar[r]^f_\simeq \ar@{^{(}->}[d]_i &
B \ar@{^{(}->}[d]^j
\\
X \ar[r]_(.30)g &
Y=X\sqcup_A B,
}
$$
where $f$ is a weak equivalence and $i$ is the cofibration obtained by
attaching a single $\mathcal I$-cell to $A$. We show that $g$ is a weak
equivalence.

First, observe that every exterior weak equivalence between exterior spaces, which are equipped with discrete rays, induces isomorphisms on the ordinary relative homotopy groups, and hence is a classical weak equivalence. Indeed,  any map of pairs
$$
\gamma\colon(D^k,S^{k-1})\longrightarrow(Y,X),\quad k\ge 0,
$$
extends to an exterior map
$$
\Gamma\colon
(\N\timess D^k,\N\timess S^{k-1})
\longrightarrow
(Y,X).
$$
Therefore, if $\Gamma$ is exteriorly homotopic, relative to
$\N\timess S^{k-1}$, to a map with image in $X$, then restricting such a
homotopy to the copy $\{0\}\times D^k$ yields an ordinary homotopy,
relative to $S^{k-1}$, from $\gamma$ to a map with image in $X$.

In particular, as every object in $\epun$ admits a discrete ray, all of the above holds for maps in $\epun$.

When $i$ attaches a $0$-cell, the map $g$ is of the form
$$
f\sqcup\id\colon
A\sqcup\N
\longrightarrow
B\sqcup\N.
$$
Since every sphere is connected, each component of an exterior map
$$
\psi\colon\bsph^k\longrightarrow A\sqcup\N
$$
lands entirely in either B or $\mathcal{N}$, and similarly for homotopies of maps $\mathcal{S^k}\rightarrow A\sqcup\N$. If $\psi^{-1}(B)$ consists of countably many spheres, surjectivity on the exterior homotopy groups follows by surjectivity of $\pi_k^e(f)$. Otherwise, $\psi^{-1}(B)$ has only finitely many components, and  surjectiviy follows by arguing on ordinary homotopy groups as noted above. The injectivity part follows by a similar argument on homotopies.

We are thus left with the case in which $i$ attaches an $n$-cell with
$n\ge1$. It suffices to check that for every $k\ge 0$ and every discrete ray 
$\alpha\colon\N\to X$, every exterior map of pairs under $\alpha$
$$
\varphi\colon
(\bdisk^k,\bsph^{k-1})
\longrightarrow
(Y,X)
$$
is exteriorly homotopic, relative to $\bsph^{k-1}$, to a map with image
contained in $X$.

Write
$$
X=A\sqcup_{\bsph^{k-1}}\bdisk^k,
\qquad
Y=B\sqcup_{\bsph^{k-1}}\bdisk^k,
$$
and let $C$ be the set of centres of the string of disks in
$\bdisk^k$. Define the open subsets
$$
X_1=X\setminus C,
\qquad
X_2=\bdisk^k\setminus\bsph^{k-1},
\qquad
Y_1=Y\setminus C,
\qquad
Y_2=X_2,
$$
covering $X$ and $Y$ respectively. 

After a finite subdivision on each disk
we may assume that $\varphi$ decomposes into pieces, each landing either
in $Y_1$ or in $Y_2=X_2$. If there are countably many pieces landing in
$Y_1$, the weak equivalence
$$
g|_{X_1}\colon X_1\longrightarrow Y_1
$$
allows us to replace them, relative to their boundaries, by pieces with
values in $X_1$. If only finitely many pieces land in $Y_1$, the
previous observation allows us to argue instead by means of ordinary
homotopy groups.

Now let
$
i\colon A\longrightarrow X
$
be a general relative $\mathcal I$-cell complex. We claim that
$$
g\colon X\longrightarrow X\sqcup_A B
$$
is a weak equivalence. In fact, since the domains of the generating cofibrations are small
relative to $\mathcal I$-cell complexes, every exterior map
$$
\bsph^n\longrightarrow X
$$
and every exterior homotopy between such maps factors through a finite
relative subcomplex of $X$. The same holds for $X\sqcup_A B$. Hence, to
prove that for every $\alpha\colon\N\to X$
$$
\pi_n^e(g,\alpha)\colon
\pi_n^e(X,\alpha)
\longrightarrow
\pi_n^e(X\sqcup_A B,g\alpha)
$$
is an isomorphism, it suffices to verify the assertion on finite
relative subcomplexes, where it follows by a finite iteration of the
single-cell case.

Finally, since cofibrations are retracts of relative
$\mathcal I$-cell complexes and weak equivalences are closed under
retracts, the result follows for arbitrary cofibrations.
\end{proof}

\section{Cofibrant objects}\label{cofibrantes}

Cofibrant objects play a distinguished role in our context. Our next goal is to obtain a geometric description of these objects and to determine whether the spaces arising naturally in exterior and proper homotopy theory belong to this class. We shall show that they encompass a large family of spaces occurring naturally in exterior and proper homotopy theory. For notational convenience, we denote by $\cofi=\E^{\mathrm{cof}}$ and $\cofi_*= \E_*^{\mathrm{cof}}$ the full subcategories of cofibrant objects in $\E$ and $\E_*$, respectively.

\begin{proposition}\label{existerayo}
For every cofibrant object $X\in\cofi$  there exists an exterior map
$$
p_X\colon X\longrightarrow\ray.
$$
\end{proposition}

\begin{proof}
Since $X$ is cofibrant, the map $\varnothing\to X$ is a retract of a relative $\mathcal I$-cell complex
$
\varnothing\to Y
$.
Hence it suffices to prove the result for $Y$.

We proceed by transfinite induction along a cellular filtration
$$
\varnothing=Y_0\longrightarrow Y_1\longrightarrow\cdots
\longrightarrow Y_\lambda\longrightarrow\cdots
\longrightarrow Y,
$$
where, for each successor ordinal $\lambda+1$, there is a pushout diagram
$$
\xymatrix{
\bsph^{n-1}\ar[r]^-{f_\lambda}\ar[d]
&
Y_\lambda\ar[d]
\\
\bdisk^n\ar[r]
&
Y_{\lambda+1}.
}
$$
At the first non-trivial stage, $Y_1=\bdisk^n$ and we take the projection
$
p_1\colon\bdisk^n\to\ray$, 
$p_1(k,x)=k$.

Assume that an exterior map
$
p_\lambda\colon Y_\lambda\to\ray
$
has been constructed, and set
$$
\varphi=p_\lambda f_\lambda\colon\bsph^{n-1}\longrightarrow\ray, \quad\varphi_k=\varphi|_{\{k\}\times S^{n-1}}, \quad k\ge 0.
$$
Since $\ray$ is convex, each $\varphi_k$ extends to a continuous map
$
\psi_k\colon\{k\}\times D^n\to\ray
$
such that
$
\im\psi_k\subset
[\min\varphi_k,\max\varphi_k]$.
These extensions define a continuous map
$$
\psi\colon\bdisk^n=\N\timess D^n\longrightarrow\ray.
$$
To see that $\psi$ is exterior, let $[0,m]\subset\ray$. Since $\varphi$ is exterior, there exists $N\ge0$ such that
$$
\varphi^{-1}[0,m]
\subset
\bigcup_{k=0}^N\{k\}\times S^{n-1}.
$$
It follows that, whenever
$
\min\varphi_k>m,
$
one also has
$
\psi_k(D^n)\cap[0,m]=\varnothing.
$
Hence
$$
\psi^{-1}[0,m]
\subset
\bigcup_{k=0}^N\{k\}\times D^n,
$$
which is compact and $\psi$ is exterior.

Since $\psi$ extends $\varphi$, the pushout above defines an exterior map
$
p_{\lambda+1}\colon Y_{\lambda+1}\to\ray
$
extending $p_\lambda$. These map induce the morphism
$
p=\varinjlim_\lambda p_\lambda\colon
Y\to\ray$.
\end{proof}

Now we develop practical criteria for constructing cofibrations in $\E$ and $\E_*$, and in particular objects in $\cofi$ and $\cofi_*$, based on controlled cell attachment constructions, and within the proper category.

\begin{proposition}\label{ray_extension}
Let $Y\in \PP$ be equipped with a ray. 
Then the ordinary attachment of an $n$-cell
$$
i\colon Y \hookrightarrow Y\sqcup_{S^{n-1}} D^n,
$$
where the resulting space is endowed with the cocompact externology, is a cofibration in $\E$. Furthermore, if the given ray in $Y$ admits a retraction, then $i$ is a cofibration in $\E_*$.
\end{proposition}

\begin{proof}
Denote by $\alpha\colon\ray\to Y$ the ray in $Y$, and let
$f\colon S^{n-1}\to Y$ be the attaching map.
Define an exterior map
$$
\varphi\colon\bsph^{n-1}\longrightarrow Y
$$
by setting
$
\varphi|_{\{0\}\times S^{n-1}}=f$, 
and, on each copy $\{k\}\times S^{n-1}$ with $k\ge1$, by taking
$\varphi$ to be the constant map with value $\alpha(k)$.
Since $\alpha$ is exterior, so is $\varphi$. Thus
$$
Y\sqcup_\varphi\bdisk^n
$$
is a well-defined exterior pushout, obtained from $Y$ by attaching one
copy of $D^n$ along $f$ and, for each $k\ge1$, an additional copy of
$S^n$ at the point $\alpha(k)$.

Consider the inclusion
$$
j\colon Y\sqcup_f D^n\longrightarrow Y\sqcup_\varphi\bdisk^n,
$$
which sends $D^n$ to the distinguished disk. Clearly, $j$ is exterior,
that is, proper. Moreover, $j$ admits a continuous and proper retraction
$$
r\colon Y\sqcup_\varphi\bdisk^n\longrightarrow Y\sqcup_f D^n,
$$
which is the identity on $Y$ and on the distinguished disk, and
collapses each extra $n$-sphere to its attaching point.

Therefore, the ordinary cell attachment
$$
i\colon Y\longrightarrow Y\sqcup_f D^n
$$
is a retract of the map
$$
Y\longrightarrow Y\sqcup_\varphi\bdisk^n,
$$
which is a relative $\mathcal I$-cell complex. Hence $i$ is a
cofibration in $\E$.

Now assume that the ray $\alpha\colon\ray\to Y$ admits a retraction $p$.
Since $\ray$ is contractible, the composite
$$
S^{n-1}\xrightarrow{f}Y\xrightarrow{p}\ray
$$
extends to a map $D^n\to\ray$. Hence $p$ extends to a continuous map
$$
 Y\sqcup_f D^n\longrightarrow\ray,
$$
which is a retraction of the ray
$$
\ray\stackrel{\alpha}{\longrightarrow}Y
\stackrel{i}{\longrightarrow}Y\sqcup_f D^n.
$$
Both maps are continuous and proper. Thus $i$ is a morphism in $\E_*$ and,
since it is a cofibration in $\E$, it is also a cofibration in $\E_*$.
\end{proof}

\begin{rem}\label{nofinito}
Observe that the argument used in the proof of the previous proposition extends formally to infinite cell attachments. However, although each finite stage carries the cocompact externology, this is no longer true for the resulting exterior space. Indeed, it is obtained as a transfinite composition of pushouts of cocompact exterior spaces, and cocompactness is not preserved under such constructions in general.
\end{rem}

A suitable version of the infinite case can, however, be treated separately as follows.

\begin{proposition}\label{local}
Let $Y\in\PP_\infty$. Then the ordinary attachment of a countable locally finite family of $n$-cells
$$
i\colon Y\hookrightarrow Y\sqcup_{\sqcup_{k\ge 0} S^{n-1}_k}(\sqcup_{k\ge 0} D^n_k),
$$
where the resulting space is endowed with the cocompact externology, is a cofibration in $\E$. 

If, moreover, $Y\in \E_*$, then $i$ is a morphism in $\E_*$ and therefore a cofibration in $\E_*$.
\end{proposition}

\begin{proof}
Let
$
\{f_k\colon S^{n-1}_k\to Y\}_{k\ge0}
$
be the family of attaching maps. Define
$$
\varphi\colon\bsph^{n-1}=\N\timess S^{n-1}\longrightarrow Y,
\qquad
\varphi|_{\{k\}\times S^{n-1}}=f_k.
$$
The map $\varphi$ is exterior by the local finiteness assumption, since only finitely many spheres meet any given compact subset of $Y$.

By construction, there is an exterior homeomorphism
$$
Y\sqcup_\varphi\bdisk^n
\cong
Y\sqcup_{\sqcup_{k\ge0}S^{n-1}_k}
(\sqcup_{k\ge0}D^n_k),
$$
where the latter is endowed with the cocompact externology. Therefore the inclusion
$i$
is a cofibration in $\E$.

Assume now that $Y\in\E_*$, with section and retraction
$
\ray\xrightarrow{\alpha}Y\xrightarrow{p}\ray
$,
and denote
$$
X=
Y\sqcup_{\sqcup_{k\ge0}S^{n-1}_k}
(\sqcup_{k\ge0}D^n_k).
$$
For each $k\ge0$, consider the composite
$$
\varphi_k\colon
S^{n-1}_k\xrightarrow{f_k}Y\xrightarrow{p}\ray.
$$
Since $\ray=[0,\infty)$ is contractible and convex, $\varphi_k$ extends to a continuous map
$
\widetilde\varphi_k\colon D^n_k\to\ray
$
such that
$
\widetilde\varphi_k(D^n_k)\subset
[\min\varphi_k,\max\varphi_k].
$
Since
$$
\widetilde\varphi_k|_{S_k^{n-1}}=p\circ f_k,
$$
the universal property of the pushout yields a continuous extension
$$
\widetilde p\colon
X\to\ray
$$
of $p$.
To see that $\widetilde p$ is exterior, let $[0,m]\subset\ray$.
If
$$
\widetilde\varphi_k(D_k^n)\cap[0,m]\neq\varnothing,
$$
then $\min\varphi_k\le m$, and hence
$$
f_k(S_k^{n-1})\cap p^{-1}[0,m]\neq\varnothing.
$$
Since $p^{-1}[0,m]$ is compact and the family of attached cells is
locally finite, this occurs for only finitely many $k$. Therefore,
$\widetilde p^{-1}[0,m]$ is contained in the union of
$p^{-1}[0,m]$ and finitely many attached disks. This union is compact,
and $\widetilde p^{-1}[0,m]$ is closed in it. Hence
$\widetilde p^{-1}[0,m]$ is compact, and therefore $\widetilde p$ is
exterior. Thus
$$
\ray\xrightarrow{\alpha}Y\xrightarrow{i}X\xrightarrow{\widetilde p}\ray
$$
is a retractive exterior space. Hence $i$ is a morphism in $\E_*$ and, since it is a cofibration in $\E$, also a cofibration in $\E_*$.
\end{proof}

A similar observation to that in Remark~\ref{nofinito} prevents us from extending the previous result to infinite-dimensional cell complexes.

Applying Propositions~\ref{ray_extension} and~\ref{local} inductively, starting from the ray, we immediately obtain the following.

\begin{corollary}\label{cw_cofibrant}
Let $(X,\ray)$ be a connected, finite-dimensional, locally finite, relative cell complex in the usual topological sense, with at most a countable number of cells in each dimension. Then, endowed with the cocompact externology,  $X\in \cofi_*$.\hfill$\square$
\end{corollary}

\begin{rem}\label{ejemplos}

(i) The above result applies to several familiar classes of spaces. For instance, it is well known that open smooth manifolds and PL-manifolds admit locally finite triangulations with countably many simplices in each dimension. 
Therefore, whenever such a triangulation can be chosen to be built over a ray, these spaces fall into our framework. In particular, the Euclidean spaces $\mathbb{R}^n$ belong to $\cofi_*$ for $n\ge 1$.

\smallskip

(ii)
The assumption that there are at most countably many cells in each dimension is automatic if $X\in \PP_\infty$. Indeed, since $X$ is $\sigma$-compact, it is the union of compact subspaces
$
X=\sqcup_{m\ge 0} K_m
$. If $X$ is locally finite, each compact subset $K_m$ meets only finitely many $n$-cells. Hence the set of $n$-cells is a countable union of finite sets, and therefore countable.
\end{rem}

\section{Spectra in the pointed exterior category}

According to the general framework in \S\ref{prespec}, the construction of spectra in $\E_*$ requires a Quillen pair of endofunctors. In view of Definition~\ref{cilpath}, the natural choice is:

\begin{definition}\label {suspen}
 Let $X\in \epun$. The \emph{cone} of $X$ is defined as the pushout  in $\E_*$
$$
\xymatrix{
X \otimes_\ray \Delta[0] \ar[r] \ar[d] & X\otimes_\ray \Delta[1] \ar[d] \\
\ray \ar[r] & CX,
}
$$
where the top map is induced by one of the face inclusions $\Delta[0]\hookrightarrow \Delta[1]$.

Accordingly, the \emph{suspension} of $X$ is defined as the pushout in $\E_*$
$$
\xymatrix{
X \otimes_\ray \partial\Delta[1] \ar[r] \ar[d] & X{\otimes_\ray}\Delta[1] \ar[d] \\
\ray \ar[r] & \Sigma X,
}
$$
where the top horizontal map is again induced by the inclusion. 

Dually,  
the \emph{loop space} of $X$ is given by the pullback in $\E_*$
$$
\xymatrix{
\Omega X \ar[r] \ar[d] & X^{\Delta[1]}_\ray \ar[d] \\
\ray \ar[r] & X^{\partial\Delta[1]}_\ray,
}
$$
where again the right vertical map is induced by $\partial\Delta[1]\hookrightarrow \Delta[1]$. The object $\Omega X$ can be interpreted as the space of loops in $X$ based at points of the distinguished ray, varying continuously along it.
\end{definition}
To avoid extra notation, we have chosen to denote the suspension and loop functors in $\epun$  as in the classical context, since no ambiguity will occur.

\medskip

As it follows from the  general context settled in Theorem \ref{sinpuntear} and Corollary \ref{coropunsin}:

\begin{corollary}\label{adjuncion}
The adjunction $\Sigma\dashv\Omega$ is a Quillen pair.\hfill$\square$
\end{corollary}

In the sequel we will make use of the continuous version of the Brown-Grossman spheres

\begin{definition}\label{pesferac}  Given $n\ge 0$,  the {\em continuous $n$th Brown-Grossman sphere} or simply the {\em continuous exterior sphere} is the exterior space consisting of attaching an $n$-sphere to each integer of the ray. 
\begin{figure}[H]
\centering
\includegraphics[height=18mm]{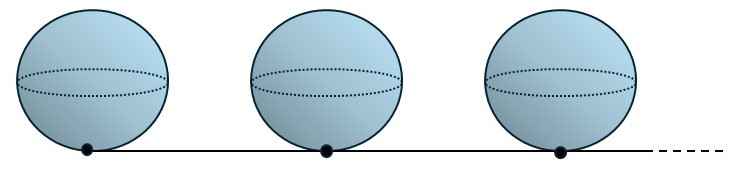}
\end{figure}
That is, 
$$
\bsphc^n=\br_+\cup\bigl( \cup_{k=0}^\infty\, S^n\times\{k\}\bigr)/\bigl(k\sim (x_0,k)\bigr)
$$
endowed with the cocompact externology, and where $x_0$ is the chosen base point in the sphere $S^n$. Equivalently $\bsphc^n\in\eunder$ arises as the exterior pushout of cocompact exterior spaces
$$
\xymatrix{
\mathbb{N}\ar[r]^(.42)\xi\,\, \ar@{^{(}->}@<-2pt>[d]_\iota& \bsph^n \ar[d]
 \\
\ray\ar[r]_\eta&
\bsphc^n}
$$
where  $\xi(m)=(m,x_0)$. The continuous exterior sphere is an object of $\epun$ equipped with the ray  $\eta$ and the retraction $\bsphc^n\to\ray$ induced by $\id_\ray$ and the projection $\bsph^n\to\ray$, $(n,x)\mapsto n$. 
\end{definition}

\begin{rem}\label{continuog}
Given  $(X,\alpha)\in \eunder$,  the universal property of the pushout  yields natural identifications
$$
\Hom_{\eunder}(\bsphc^n,X)\cong \Hom_{\E^\N}\bigl((\bsph^n,\xi),(X,\alpha\iota)\bigr)
$$
which induce a bijection 
$$
[\bsphc^n,X]^\ray\cong [\bsph^n,X]^\N=\pi_n^e(X,\alpha|_\N).
$$
\end{rem}

The following is then immediate from the definitions.

\begin{proposition}\label{suses} $\sus\bsphc^n\cong \bsphc^{n+1}$ for any $n\ge 0$.\hfill$\square$
\end{proposition}

 Henceforth, for any object $X \in \E_*$ and any $n \ge 0$, we write $\pi_n^e(X)$ for $\pi_n^e(X,\alpha|_\N)$, where $\alpha\colon \ray \to X$ is  the section ray of $X$.

\begin{proposition}\label{homolup}
For any fibrant object $X\in \E_*$ and any $n\ge 0$, there is a natural isomorphism
$$
\pi_n^e(\Omega X)\cong \pi_{n+1}^e(X).
$$
\end{proposition}

\begin{proof}
Since $\Omega$ is a right Quillen functor, $\Omega X$ is fibrant while $\bsphc^n$ is clearly a cofibrant object of $\epun$. Hence, by Remark \ref{continuog}, Propositions~\ref{rectificacion} and Proposition~\ref{suses}, together with  Corollary \ref{adjuncion}, we obtain natural isomorphisms:
$$
\begin{aligned}
\pi_n^e(\Omega X)&=[\bsphc^n,\Omega X]^\ray\cong [\bsphc^n,\Omega X]_{\E_*}\cong [\Sigma \bsphc^n,X]_{\E_*}\\&\cong [\bsphc^{n+1},X]_{\epun}\cong [\bsphc^{n+1},X]^\ray\cong\pi_{n+1}^e(X).
\end{aligned}
$$
\end{proof}

We now have all the necessary ingredients to develop spectra in the homotopy category:

\begin{definition}\label{espectroe}
An \emph{exterior spectrum} is a spectrum in $\E_*$ with respect to the Quillen adjunction of endofunctors $\Sigma \dashv \Omega$. In other words, it consists of a sequence $X = \{X_n\}_{n \geq 0}$ of exterior spaces in $\E_*$ together with a family of \emph{structure maps} also  in $\epun$
$$
\sigma_n \colon \Sigma X_n \longrightarrow X_{n+1}, \quad \text{or equivalently,} \quad \eta_n \colon X_n \longrightarrow \Omega X_{n+1}, \quad n \geq 0.
$$

A \emph{map $f \colon X \to Y$ of exterior spectra} is a collection $\{f_n \colon X_n \to Y_n\}_{n \geq 0}$ of maps in $\epun$ such that, for all $n \geq 0$, the following diagram commutes:
$$
\xymatrix{
\Sigma X_n \ar[r]^{\Sigma f_n} \ar[d]_{\sigma_n} & \Sigma Y_n \ar[d]^{\sigma_n} \\
X_{n+1} \ar[r]_{f_{n+1}} & Y_{n+1}
}
\quad \text{or equivalently,} \quad
\xymatrix{
X_n \ar[r]^{f_n} \ar[d]_{\eta_n} & Y_n \ar[d]_{\eta_n} \\
\Omega X_{n+1} \ar[r]_{\Omega f_{n+1}} & \Omega Y_{n+1}.
}
$$

We denote the resulting category by $\spece$. From now on, unless otherwise stated, we will simply refer to an exterior spectrum as a \emph{spectrum}.
\end{definition}

\begin{example} 
 Given $X\in\epun$, the \emph{exterior suspension spectrum} $\Sigma X\in \spece$ is defined by $(\Sigma X)_n=\Sigma^n X$, with identity structure maps. 
As a particular case, taking $X=\bsphc^0$, Proposition \ref{suses} yields the \emph{exterior sphere spectrum} $\mathcal{S}_\ray=\{\bsphc^n\}_{n\ge 0}$.
\end{example}

What follows is a direct translation into our context of the general results collected in \S\ref{prespec}: 

\begin{proposition}
The Quillen pair $\sus\dashv\lup$ prolongs to adjoint functors $\suse\dashv\lupe$ in $\spece$. Furthermore, for each $k\ge0$, the evaluation functor
$$
 \eva_k\colon \spece\longrightarrow \E_*,\quad X\mapsto X_k,
$$
 is right adjoint to
$$
 \sus^{\infty-k}\colon \E_*\longrightarrow\spece,\quad (\sus^{\infty-k}X)_n=\begin{cases} \sus^{n-k}X&\text{if $n\ge k$},\\ \,\,\,\ray&\text{otherwise}.\end{cases}
$$
\hfill$\square$
\end{proposition}

The cellular, left proper, and cofibrantly generated \emph{projective model structure} on $\spece$ is defined by declaring the projective weak equivalences and projective fibrations to be the levelwise weak equivalences and fibrations in $\epun$. With respect to this structure, the  pairs of functors $\suse \dashv \lupe$ and $\sus^{\infty-k} \dashv \eva_k$, for $k \geq 0$, are Quillen adjunctions.

The \emph{stable model structure}, which preserves cellularity and left properness on $\spece$, is obtained as the Bousfield localization of the projective structure with respect to the set of morphisms given in (\ref{equibous}). In this structure, the pair $\suse \dashv \lupe$ is a Quillen equivalence, and the fibrant objects are the \emph{$\Omega$-spectra}. That is, exterior spectra $X$ which are levelwise fibrant and such that the adjoints  $X_n \xrightarrow{\simeq} \lup X_{n+1}$ of the structure maps are weak equivalences in $\epun$.

An {\em stable exterior stable equivalence} is a weak equivalence in the stable model structure. As in the classical case, we characterize exterior stable equivalences as follows.

\begin{definition}\label{stableq}
Let $X\in\spece$ be an exterior spectrum. 
 For any
$k\in \mathbb Z$, the $k$-th \emph{exterior stable homotopy group}
of $X$ is defined by
$$
\pi_k^{s,e}(X)
=
\varinjlim_n
\pi_{k+n}^{e}(X^{fib}_n),
$$
where $X^{fib}$ is a fibrant replacement of $X$ in the stable model
structure, that is, an exterior $\Omega$-spectrum stably weakly
equivalent to $X$, and the colimit is taken over the morphisms
$$
 \pi_{k+n}^{e}(X^{fib}_n)
\stackrel{\pi_{k+n}^{e}(\eta_n)}{\longrightarrow}
\pi_{k+n}^{e}(\Omega X^{fib}_{n+1})
\cong
\pi_{k+n+1}^{e}(X^{fib}_{n+1}).
$$
As every $\eta_n$ is a weak equivalence these maps are isomorphisms and thus,  for every $k\in\mathbb Z$
and every $n\geq 0$ with $k+n\geq 0$, the canonical map
$$
\pi_{k+n}^{e}(X^{fib}_n)
\stackrel{\cong}{\longrightarrow}
\pi_k^{s,e}(X)
$$
is also an isomorphism.
\end{definition}

\begin{theorem}\label{stableequivstablegroups}
A map $f\colon X\to Y$ of exterior spectra is a stable  equivalence
if and only if it induces  isomorphisms
$$
\pi_k^{s,e}(f):
\pi_k^{s,e}(X)\xrightarrow{\cong}\pi_k^{s,e}(Y),\quad k\in\Z.
$$
\end{theorem}

In the classical stable setting, one defines the stable homotopy groups using only the distinguished base point of the loop spaces in an $\Omega$-spectrum, since translations induce homotopy equivalences between all path components of a loop space. The analogous fact
is not immediate for exterior loop spaces, because they are defined in
terms of exterior loops. The following lemma shows that, nevertheless, all discrete rays of an
exterior loop space determine the same exterior homotopy groups.

\begin{lemma}\label{translation}
Let
$
f\colon\Omega Z\to\Omega W
$
be a morphism in $\E_*$ between fibrant exterior loop spaces. Suppose that,
for a fixed $n\ge1$, the map
$$
\pi_n^e(f,\alpha)\colon
\pi_n^e(\Omega Z,\alpha)
\stackrel{\cong}{\longrightarrow}
\pi_n^e(\Omega W,f\alpha)
$$
is an isomorphism, where $\alpha$ denotes the distinguished discrete ray
of $\Omega Z$. Then, for every discrete ray
$
\beta\colon\N\to\Omega Z
$,
the map
$$
\pi_n^e(f,\beta)\colon
\pi_n^e(\Omega Z,\beta)
\stackrel{\cong}{\longrightarrow}
\pi_n^e(\Omega W,f\beta)
$$
is also an isomorphism.
\end{lemma}

\begin{proof}
Let
$
p_Z\colon\Omega Z\to\ray$ and $
p_W\colon\Omega W\to\ray
$
be the retractions which induces continuous maps 
$$
p_Z^\N\colon
\map_\E(\N,\Omega Z)\longrightarrow\map_\E(\N,\ray)\quad\text{and}\quad
p_W^\N\colon
\map_\E(\N,\Omega W)\longrightarrow\map_\E(\N,\ray)
$$
with respect to the particular topologies given in \S\ref{prelimi2}. Furthermore,
both are Serre fibrations. Indeed, under the exponential law of Theorem \ref{expo}, their lifting problems are adjoint to
lifting problems for $p_Z$ and $p_W$ with respect to the generating
trivial cofibrations
$
\bdisk^m\longrightarrow\bcil^m$. As $\Omega Z$ and $\Omega W$ are fibrant it follows that $p_Z$ and $p_W$ are exterior fibrations and the lifting problem can be solved.

If we denote 
$
q=p_Z\beta=p_Wf\beta 
$
 the map

$$
H\colon \N\timess I\longrightarrow \ray,\quad H(k,t)=(1-t)q(k)+tk
$$
defines an exterior homotopy from $q$ to the inclusion $\iota \colon \N\hookrightarrow\ray$. Indeed, for every
$m\ge0$,
$$
H^{-1}[0,m]
\subset
\bigl(q^{-1}[0,m]\times I\bigr)
\cup
\bigl(\{0,\ldots,m\}\times I\bigr),
$$
which is compact. Hence, by Theorem~\ref{expo}, $H$ determines a path
from $q$ to $\iota$ in $\map_\E(\N,\ray)$.
Transport along this path yields homotopy equivalences between the
corresponding fibers of $p_Z^\N$ and $p_W^\N$
$$
F_Z(q)\simeq F_Z(\iota),
\qquad
F_W(q)\simeq F_W(\iota),
$$
natural, up to homotopy, with respect to the map induced by $f$. In
particular, $\beta$ and $f\beta$ are transported to points
$$
\beta'\in F_Z(\iota),
\qquad
f\beta'\in F_W(\iota).
$$

Applying $\map_\E(\N,-)$ to the pullback defining the exterior loop
space and using Theorem~\ref{expo}, the fibers $F_Z(\iota)$ and
$F_W(\iota)$ are naturally identified with classical loop spaces.
Consequently, translations induce homotopy equivalences
$$
(F_Z(\iota),\beta')\simeq(F_Z(\iota),\alpha),
\quad
(F_W(\iota),f\beta')\simeq(F_W(\iota),f\alpha),
$$
and these equivalences are natural with respect to $f$.

It follows that the map induced by $f$ on the homotopy groups based at
$\beta$ is identified with the corresponding map based at $\alpha$.
The latter is an isomorphism by hypothesis. Finally, using the natural
identifications
$$
\pi_n^e(\Omega Z,\gamma)
\cong
\pi_n\bigl(\map_\E(\N,\Omega Z),\gamma\bigr)
$$
and similarly for $\Omega W$, we conclude that
$
\pi_n^e(f,\beta)
$
is an isomorphism.
\end{proof}
\begin{proof}[Proof of Theorem \ref{stableequivstablegroups}]
By definition, $f$ induces  isomorphisms on all exterior stable homotopy groups if and only if
the chosen fibrant replacement $f^{fib}$ does. Therefore, it suffices
to prove the result for maps between exterior $\Omega$-spectra.

Assume first that $f\colon X\to Y$ is a stable  equivalence between
exterior $\Omega$-spectra and recall that such a map is always a levelwise  equivalence. Hence, each map
$
f_n\colon X_n\to Y_n
$
induces
isomorphisms on all exterior homotopy groups. Therefore, by Definition \ref{stableq},  $\pi_k^{s,e}(f)$ is an isomorphism for every $k\in\mathbb Z$.

Conversely, suppose that $f\colon X\to Y$ is a map of exterior
$\Omega$-spectra inducing  isomorphisms on all exterior stable homotopy groups. Then, for every
$n\geq 0$ and every $m\geq 0$, taking $k=m-n$, we have  commutative
diagrams
$$
\xymatrix{
\pi_m^e(X_n) \ar[r]^{\pi_m^e(f_n)} \ar[d]_{\cong} &
\pi_m^e(Y_n) \ar[d]^{\cong}
\\
\pi_{m-n}^{s,e}(X) \ar[r]_{\cong} &
\pi_{m-n}^{s,e}(Y).
}
$$
It follows that 
$
\pi_m^e(f_n)$, with respect to the discrete ray induce by the section of $X_n$,
is an isomorphism for $m,n\ge 0$.

Now, for $m=0$, and using the natural identification
$
\pi_0^e(X_n,\beta)\cong[\N,X_n]
$ for any discrete ray $\beta$, 
this shows that
$$
[\N,f_n]\colon[\N,X_n]\longrightarrow[\N,Y_n]
$$
is a bijection.

Now let $m\ge1$ and let
$
\beta\colon\N\to X_n
$
be any discrete ray. Since $X$ and $Y$ are $\Omega$-spectra, their
structure maps fit into a commutative square
$$
\xymatrix{
X_n \ar[r]^{f_n} \ar[d]_{\eta_n^X} &
Y_n \ar[d]^{\eta_n^Y}
\\
\Omega X_{n+1} \ar[r]_{\Omega f_{n+1}} &
\Omega Y_{n+1}.
}
$$
where the vertical maps are weak equivalences and therefore induce
isomorphisms
$$
\pi_m^e(X_n,\beta)
\stackrel{\cong}{\longrightarrow}
\pi_m^e(\Omega X_{n+1},\eta_n^X\beta),\qquad 
\pi_m^e(Y_n,f_n\beta)
\stackrel{\cong}{\longrightarrow}
\pi_m^e(\Omega Y_{n+1},\eta_n^Yf_n\beta).
$$
On the other hand, since $
\pi_{m+1}^e(f_{n+1})
$
is an isomorphism at the distinguished discrete ray,  it follows by the natural
isomorphism of Proposition~\ref{homolup} that
$
\pi_m^e(\Omega f_{n+1})
$
is also an isomorphism at the distinguished discrete ray of
$\Omega X_{n+1}$. Lemma~\ref{translation}, applied to the map
$\Omega f_{n+1}$ and the discrete ray $\eta_n^X\beta$, now yields an
isomorphism
$$
\pi_m^e(\Omega X_{n+1},\eta_n^X\beta)
\stackrel{\cong}{\longrightarrow}
\pi_m^e(\Omega Y_{n+1},\eta_n^Yf_n\beta).
$$
Therefore, in the commutative diagram
$$
\xymatrix{
\pi_m^e(X_n,\beta)
\ar[rr]^{\pi_m^e(f_n)}
\ar[d]_{\cong}
&&
\pi_m^e(Y_n,f_n\beta)
\ar[d]^{\cong}
\\
\pi_m^e(\Omega X_{n+1},\eta_n^X\beta)
\ar[rr]_{\cong}^{\pi_m^e(\Omega f_{n+1})}
&&
\pi_m^e(\Omega Y_{n+1},\eta_n^Yf_n\beta),
}
$$
the upper horizontal map is also an isomorphism.

Thus, for every $n\ge0$, the map $f_n$ induces a bijection on discrete
rays and isomorphisms on all exterior homotopy groups based at every
discrete ray. Hence $f_n$ is an exterior weak equivalence. Therefore
$f$ is a levelwise weak equivalence between exterior $\Omega$-spectra,
and consequently a stable equivalence.
\end{proof}

As any two fibrant replacements of an exterior spectrum are connected by
a zig-zag of stable equivalences, Theorem~\ref{stableequivstablegroups}
immediately yields the following.

\begin{corollary}
The exterior stable homotopy groups are well defined. That is, they do
not depend on the choice of stable fibrant replacement.\hfill$\square$
\end{corollary}

\begin{rem}
Following the classical construction, one could define the exterior
stable homotopy groups of an arbitrary spectrum $X\in\spece$ directly by
$$
\pi_k^{s,e}(X)=
\varinjlim_n\pi_{k+n}^e(X_n),
$$
without passing to a fibrant replacement. However, in order to prove
Theorem~\ref{stableequivstablegroups}, one would first have to show that
these groups are invariant under stable equivalences. This can be
achieved by constructing a spectrification functor and proving that it
provides a stable fibrant replacement. Since this requires a
considerable amount of additional technical work and is not needed in
the sequel, we have chosen the equivalent definition via fibrant
replacements.
\end{rem}

 We denote by $\spec$ the category of spectra of pointed simplicial sets, and by $\pi_*^{st}$ the stable homotopy groups of an object in this category.

\begin{theorem}\label{homoto}
There is a bifunctor
$$
\calf\colon \E_*^{op}\times \spece\longrightarrow \spec
$$
such that, for every cofibrant object $X\in\cofi_*$, and any $\Omega$-spectrum $Y\in \spece$,
$$
\pi_k^{st}\calf(X,Y)\cong [X,Y_k]_{\epun}\cong[X,Y_k]^\ray\quad \text{for all $k\in\Z$.}
$$
\end{theorem}
We call $\calf$ the {\em function spectrum}. Here, for $k<0$, we define $Y_k=\Omega^{-k}Y_0$. For the proof  we need some preparation.

\begin{definition}\label{esmash} Let $X \in \E_*$ and $K \in \catss_*^f$, with basepoint $*$.
\begin{enumerate}
\item  The {\em pointed tensor} $X\wedge_\ray K$, or {\em smash product}, of $X\in \E_*$ by $K$ is defined as the pushout in $\E_*$
$$
\xymatrix{
X\ar[r] \ar[d] &
X\otimes_\ray K \ar[d] \\
\ray \ar[r] &
X\wedge_\ray K,
}
$$
where the top horizontal arrow is the map $X\cong X\otimes_\ray * \to X\otimes_\ray K$ induced by the basepoint inclusion $*\to K$, and the left vertical map is the unique morphism $X\cong X\otimes_\ray *\to\ray$ in $\E_*$.

\item Dually, the \emph{pointed cotensor}
$X_*^{K}\in \E_*$ is defined as the pullback in $\E_*$ 
$$
\xymatrix{
(X^K_\ray)_* \ar[r] \ar[d] &
X^K_\ray \ar[d] \\
\ray \ar[r] &
X ,
}
$$
where the bottom horizontal arrow is the map $\ray\to X\cong X^*_\ray$ and the right vertical map is the morphism $X^K_\ray \to X^*_\ray\cong X$ induced again by the inclusion of the basepoint of $K$. 
\end{enumerate}

\end{definition}

\begin{rem} Note that the pointed tensor may also be regarded as the
 pushout in $\E$,
$$
\xymatrix{
(\ray\timess |K|) \cup_{\ray} X\ar[r] \ar[d] &
X \timess |K| \ar[d] \\
\ray \ar[r] &
X \wedge_\ray K,
}
$$
where the  upper left  corner is the pushout in $\E$ of the inclusions
$
\ray\cong \ray\timess *\to \ray\timess |K|$ and $\ray\to X$ and both,  
and both the top horizontal and left vertical maps are induced by the projection $K\to *$ and the inclusion $*\to K$.

On the other hand, $X^K_*$ can be constructed as the pullback in $\E$
$$
\xymatrix{
X^K_* \ar[r] \ar[d] &
X^{|K|} \ar[d] \\
\ray \ar[r] &
\ray^{|K|}\times_\ray X,
}
$$
where the lower right corner is the pullback in $\E$ of the maps
$$
\ray^{|K|}\to \ray^*\cong \ray
\quad \text{and} \quad
X\to \ray,
$$
and the lower horizontal and right vertical maps are induced by the basepoint inclusion $*\to K$ and the projection $K\to *$. 
 
In each case, the resulting object inherits a canonical structure over and under $\ray$, i.e. it is an object of $\E_*$.

\end{rem}

With respect to these constructions we have:

\begin{proposition}
The category $\E_*$  is
enriched, tensored and cotensored over $\catss_*$, and these structures
are compatible with the standard model structure on $\catss_*$ in the
usual Quillen sense.
\end{proposition}

\begin{proof}
This follows from the general fact that if $\mathcal C$ is a simplicial model category and $a\in \mathcal C$, then the retractive category $\mathcal C_a$ inherits a pointed simplicial model structure, which is naturally tensored and cotensored over $\catss_*$. See, for instance, \cite[Thm.~4.2.19]{ho0}. The present situation corresponds to the case $\mathcal C=\E$ and $a=\ray$.
\end{proof}

 In particular:
 
\begin{corollary}\label{coropun} Let $X,Y \in \E_*$ and  $K,L \in \catss_*$. Then:

(i) There are natural
isomorphisms of pointed simplicial sets,
$$
\Map_{\E_*}(X \wedge_\ray K, Y)
\cong
\Map_{\catss_*}\bigl(K,\Map_{\E_*}(X,Y)\bigr)
\cong
\Map_{\E_*}(X, Y^K_*).
$$

(ii) There are natural
isomorphisms in $\E_*$,
$$\begin{aligned}
X \wedge_\ray \Delta[0] &\cong \ray \cong X_*^{\Delta[0]},\\
X \wedge_\ray (K \wedge L) &\cong (X \wedge_\ray K)\wedge_\ray L,\\
X_*^{K \wedge L} &\cong (X_*^K)^L_*.
\end{aligned}
$$

(iii) The adjunction
$
(-)\wedge_\ray K \dashv (-)^K_*
$
is a Quillen pair.\hfill$\square$
\end{corollary}

\begin{proposition}\label{lupita}
For every $X\in \E_*$, and any $n\ge 1$, there are natural isomorphisms
$$
\Sigma^n X \cong X\wedge_\ray S^n,
\qquad
\Omega^n X \cong X_*^{S^n}.
$$
\end{proposition}

\begin{proof} We address first the case $n=1$. Since $S^1$ is the pushout in $\catss$
$$\xymatrix{
\partial\Delta[1] \ar[r] \ar[d] & \Delta[1] \ar[d] \\
\Delta[0] \ar[r] & S^1,
}
$$
the functor $X{\otimes_\ray}{-}$ is 
 left adjoint (see (i) of Corollary \ref{coropunsin}), and $X\otimes_\ray \Delta[0]\cong X$, it follows that the diagram 
$$
\xymatrix{
X\otimes_\ray \partial\Delta[1] \ar[r] \ar[d] &
X\otimes_\ray \Delta[1] \ar[d] \\
X \ar[r] &
X\otimes_\ray S^1.
}
$$
is a pushout in $\epun$.

Composing it with the pushout defining $X\wedge_\ray S^1$ (see Definition \ref{esmash}), we obtain a pushout in $\E_*$
$$
\xymatrix{
X\otimes_\ray \partial\Delta[1] \ar[r] \ar[d] &
X\otimes_\ray \Delta[1] \ar[d] \\
\ray \ar[r] &
X\wedge_\ray S^1
}
$$
which is exactly the diagram defining $\Sigma X$ (see Definition \ref{suspen}).

Dually, since 
 $X_\ray^{(-)}$ is a contravariant right adjoint and  $X_\ray^{\Delta[0]}\cong X$, the diagram
$$
\xymatrix{
X_\ray^{S^1} \ar[r] \ar[d] &
X_\ray^{\Delta[1]} \ar[d] \\
X \ar[r] &
X_\ray^{\partial\Delta[1]}
}
$$
is a pullback in $\E_*$. Precomposing on the left with the pullback defining $X_*^{S^1}$, we obtain the pullback in $\E_*$ defining $\Omega X$:
$$
\xymatrix{
X_*^{S^1} \ar[r] \ar[d] &
X_\ray^{\Delta[1]} \ar[d] \\
\ray \ar[r] &
X_\ray^{\partial\Delta[1]}.
}
$$
The general case follows easily by induction in view of  Corollary  \ref{coropun}(ii).
\end{proof}

\begin{proof}[Proof of Theorem \ref{homoto}]
Define $\calf(X,Y)\in \spece$ by
$$
\calf(X,Y)_n=\Map_{\E_*}(X,Y_n).
$$
On the other hand, the structure maps 
$$
\calf(X,Y)_n \longrightarrow \Omega\,\calf(X,Y)_{n+1}
$$
are obtained using the structure maps of $Y$, together with the natural isomorphisms given by  Corollary~\ref{coropun}(i) and Proposition~\ref{lupita}:
$$
\begin{aligned}
\Map_{\E_*}(X,Y_n)
\longrightarrow &\Map_{\E_*}(X,\Omega Y_{n+1})\cong \Map_{\E_*}\bigl(X,(Y_{n+1})_*^{S^1}\bigr)\\
&\cong \Map_{\catss_*}\bigl(S^1,\Map_{\E_*}(X,Y_{n+1})\bigr)= \Omega\,\Map_{\E_*}(X,Y_{n+1}).
\end{aligned}
$$
Next, if $Y$ is a fibrant object in the stable model structure,  each adjoint structure map
$
Y_n\xrightarrow{\simeq} \Omega Y_{n+1}
$
is a weak equivalence in $\E_*$. As a result, each adjoint structure map of $\calf(X,Y)$ is a weak equivalence of pointed simplicial sets and thus, $\calf(X,Y)$ is also an $\Omega$-spectrum in $\spec$. Hence, for every $k\in \Z$,
$$
\pi_k^{st}\calf(X,Y)\cong \pi_0\calf(X,Y)_k
=\pi_0\Map_{\E_*}(X,Y_k)\cong [X,Y_k]_{\epun},
$$
which furthermore, see Proposition~\ref{rectificacion}, it agrees with $[X,Y_k]^\ray$   since $X$ is cofibrant and each $Y_k$ is fibrant. \end{proof}

\section{Classification of exterior and proper cohomology theories}

In what follows $\cofi^{\mathrm{pair}}$ denotes the category of pairs $(X,A)$ of cofibrant objects of $\E$ where the inclusion $A\hookrightarrow X$ is also a cofibration. As usual, a single object $X\in \cofi$ is identified with the pair $(X,\varnothing)\in \cofi^{\mathrm{pair}}$.

\begin{definition}\label{noredu} An \emph{(extraordinary) exterior cohomology theory} $\calh$ consists of a collection of contravariant functors
$$
\calh^n \colon \cofi^{\mathrm{pair}} \longrightarrow \mathbf{Ab}, \qquad n \in \mathbb{Z},
$$
together with natural transformations
$$
\delta \colon \calh^n(A) \longrightarrow \calh^{n+1}(X,A),
$$
satisfying:

\medskip

{(Homotopy)}  
If $f,g \colon (X,A) \to (Y,B)$ are exterior homotopic maps of pairs, that is, there exists an exterior homotopy
$
H\colon X\timess I \to Y
$
from $f$ to $g$ such that
$
H(A\times I)\subset B$,  then
$$
\calh^*(f) = \calh^*(g) \colon \calh^*(Y,B) \longrightarrow \calh^*(X,A).
$$

\medskip

{(Exactness)}
For every pair $(X,A)\in \cofi^{\mathrm{pair}}$ the sequence
$$
\cdots \longrightarrow \calh^n(X,A) \xlongrightarrow{\calh^n(j)} \calh^n(X) \xlongrightarrow{\calh^n(i)} \calh^n(A) \xrightarrow{\delta} \calh^{n+1}(X,A) \longrightarrow \cdots
$$
is exact. Here $i\colon A\hookrightarrow X$ and $j\colon (X,\varnothing)\to (X,A)$ are the inclusions.

\medskip

{(Additivity)}  
For any family $\{(X_i,A_i)\}_{i \in I}$ of pairs in $\cofi^{\mathrm{pair}}$, the canonical map
$$
\calh^n\bigl({\textstyle\sqcup_{i\in I}} (X_i,A_i)\bigr) \xlongrightarrow{\cong} {\textstyle\prod_{i\in I}} \calh^n(X_i,A_i)
$$
is an isomorphism. Here, the coproduct is taken in $ \E$ and  
$$\sqcup_{i \in I} (X_i,A_i)=
(\sqcup_{i \in I} X_i, \sqcup_{i \in I} A_i).
$$

\medskip

{(Excision)}
Let $i\colon A\hookrightarrow X$ be a cofibration and let
$f\colon A\to B$ be any map in $\E$. The canonical morphism in
$\cofi^{\mathrm{pair}}$
$$
(X,A)\longrightarrow (X\sqcup_A B,B),
$$
induced by the pushout of $i$ along $f$, yields an isomorphism
$$
\calh^n(X\sqcup_A B,B)\xrightarrow{\cong}\calh^n(X,A).
$$
Note that, since cofibrations are stable under pushouts, the canonical
map
$$
B\hookrightarrow X\sqcup_A B
$$
is again a cofibration, and hence
$(X\sqcup_A B,B)\in\cofi^{\mathrm{pair}}$.
\end{definition}

 We have chosen to formulate the excision axiom  in terms of pushouts along cofibrations, rather than quotients, since the ambient category $\E$ is not pointed. This is intrinsic to the homotopical structure of $\E$ and  is the natural analogue of the classical excision property for CW-pairs, where decompositions $X = A \cup B$ correspond to pushout squares.

\begin{definition}\label{cofibra} Given a cofibration $A \hookrightarrow X$ in $\epun$, we define its {\em cofiber} as the object in $\epun$ given by 
$$
X/A=X\sqcup_A \ray.
$$
\end{definition}

Although unreduced theories provide a natural framework for pairs in $\E$, the reduced theory on $\E_*$ will play a central role in what follows.

\begin{definition}\label{redu}
A \emph{(extraordinary) reduced exterior  cohomology theory} $\calhr$ is a collection of contravariant functors
$$
\calhr^n \colon \cofi_* \longrightarrow \mathbf{Ab}, \qquad n \in \mathbb{Z},
$$
satisfying:

\medskip

{(Homotopy)}  
If $f,g\colon X\to Y$ are maps in $\cofi_*$ which are exteriorly homotopic relative to $\ray$, then
$$
\calhr^*(f) = \calhr^*(g) \colon \calhr^*(Y) \longrightarrow \calhr^*(X).
$$

\medskip

{(Exactness)}  
For every cofibration $i \colon A \hookrightarrow X$ in $\cofi_*$, there is a long exact sequence
$$
\cdots \longrightarrow \calhr^n(X/A)
\xrightarrow{\calhr^n(q)}
\calhr^n(X)
\xrightarrow{\calhr^n(i)}
\calhr^n(A)
\xrightarrow{\delta}
\calhr^{n+1}(X/A)
\longrightarrow \cdots,
$$
where $q \colon X \to X/A$ denotes the canonical map.

\medskip

{(Additivity)}  
For any family $\{X_i\}_{i \in I}$ of objects in $\cofi_*$, the canonical map
$$
\calhr^n\bigl({\textstyle\bigvee_{i \in I}} X_i\bigr)
\xrightarrow{\cong}
{\textstyle\prod_{i \in I}} \calhr^n(X_i)
$$
is an isomorphism. Here $\bigvee$ denotes the coproduct in $\epun$.
\end{definition}

\begin{proposition}\label{equivred}
Every exterior extraordinary cohomology theory $\calh$ on $\E$ induces a reduced exterior extraordinary cohomology theory on $\E_*$ by setting
$$
\calhr^*(X)=\calh^*(X,\ray),\quad X\in \cofi_*.
$$
\end{proposition}
\begin{proof}
Given an exterior extraordinary cohomology theory $\calh$ on $\E$, define a reduced theory $\calhr$ on $\epun$ by
$$
\calhr^n(X)=\calh^n(X,\ray).
$$
The homotopy and additivity axioms for $\calhr$ are immediate from the corresponding axioms for $\calh$. Moreover, for every cofibration $A\hookrightarrow X$ in $\cofi_*$, the exactness axiom for $\calh^*$ together with excision yields a natural isomorphism
\begin{equation}\label{hcofibra}
\calh^n(X/A,\ray)\cong \calh^n(X,A).
\end{equation}
Indeed, $(X/A,\ray)=(X\sqcup_A \ray,\ray)$, and the canonical morphism of pairs
$$
(X,A)\longrightarrow (X\cup_A\ray,\ray)
$$
induces an isomorphism by excision.
 Hence the long exact sequence on $\calh$ for the pair $(X,A)$ identifies with
$$
\cdots \longrightarrow \calhr^n(X/A)
\longrightarrow \calhr^n(X)
\longrightarrow \calhr^n(A)
\xrightarrow{\delta}
\calhr^{n+1}(X/A)
\longrightarrow \cdots.
$$
\end{proof}

\begin{theorem}\label{repre} 
Every reduced exterior extraordinary cohomology theory $\calhr^*$ is representable by an exterior $\Omega$-spectrum $E\in\spece$. That is, for every $X\in \cofi_*$ and every $n\ge 0$, there are natural isomorphisms
$$
\calhr^n(X)\cong [X,E_n]_\epun.
$$
\end{theorem}
 For $n<0$, the groups $\calhr^n(X)$ are also determined by the suspension isomorphisms
$$
\calhr^n(X)\cong \calhr^0(\Sigma^{-n}X).
$$
\begin{proof}The additivity axiom shows that $\calhr^n$ sends arbitrary coproducts to products. Moreover, exactness implies the Brown Mayer--Vietoris condition: for every pushout square of cofibrations
$$
\xymatrix{
A \ar[r] \ar[d] & X \ar[d] \\
Y \ar[r] & X\cup_A Y ,
}
$$
the canonical map
$
\calhr^n(X\cup_A Y)\to
\calhr^n(X)\times_{\calhr^n(A)}\calhr^n(Y)
$
is surjective. This follows by applying the long exact sequences associated to the cofibrations
$$
A\hookrightarrow X
\qquad\text{and}\qquad
Y\hookrightarrow X\cup_A Y,
$$
together with the natural identification
$$ 
(X\cup_A Y)/Y\cong X/A.
$$
 Hence, 
since every object of $\cofi_*$ is a retract of a relative $\mathcal I$-cell complex, Brown's cellular argument \cite{brown2} can be carried out using the generating cofibrations in $\mathcal I$ to build the cellular objects. In fact, Brown's inductive construction adapts verbatim once generators in dimension zero are included in the initial step and thus, no connectedness  assumption is required, since the generators in dimension zero account for the connectivity data. Therefore, there exists an object $E_n\in\epun$ together with natural isomorphisms: 
$$
\calhr^n(X)\cong [X,E_n]_\epun
\qquad X\in\cofi_*.
$$

To assemble the objects $E_n$ into an exterior spectrum note first that, for each $X\in \cofi_*$, its cone $CX$ lies also in $\cofi_*$ and the canonical map $X\hookrightarrow CX$ is a cofibration. Moreover, the suspension $\Sigma X$ is naturally identified with the cofiber of this map, that is,
$
\Sigma X \cong CX / X$.
Applying the exactness axiom to this cofibration  yields natural suspension isomorphisms
$$
\calhr^n(X)\cong \calhr^{n+1}(\Sigma X)
$$
which, transported through the representability bijections, provide natural isomorphisms
$$
[X,E_n]_\epun\cong [\Sigma X,E_{n+1}]_\epun \cong [X,\mathbf \Omega E_{n+1}]_\epun.
$$
Hence, by Yoneda's lemma 
 $$ E_n\cong\mathbf \Omega E_{n+1} $$ in $\Ho(\epun)$. Choosing fibrant representatives for the $E_n$, these isomorphisms may be represented by weak equivalences $$ E_n\xrightarrow{\simeq}\Omega E_{n+1}
  $$
  and therefore the sequence $E=\{E_n\}_{n\geq0}$ forms an exterior $\Omega$-spectrum representing $\calhr^*$.

 The values of $\calhr^n$ in negative
degrees are then recovered from the suspension isomorphisms:
$$
\calhr^n(X)\cong \calhr^0(\Sigma^{-n}X), \qquad n<0.
$$
Equivalently, using repeatedly the adjunction
$\Sigma\dashv\Omega$, one obtains natural isomorphisms
$$
\calhr^n(X)\cong [X,\Omega^{-n}E_0]_\epun,
\qquad n<0.
$$
\end{proof}

As an immediate consequence of Theorem \ref{homoto}, the cohomology theory $\calhr$ is recovered as the stable homotopy groups of the mapping spectrum into $E$:

\begin{corollary}\label{cororepre} For every $X\in\cofi_*$ there are natural isomorphisms
$$
\pi_*^{st}\calf(X,E)\cong \calhr^*(X).
$$
\hfill$\square$
\end{corollary}

\begin{rem}\label{cohoho}
Every reduced extraordinary cohomology theory $\calhr$ on cofibrant objects extends naturally to a family of contravariant functors
$$
\calhr^n\colon \E_* \longrightarrow \mathbf{Ab}, \qquad n\in\mathbb{Z},
$$
by setting
$$
\calhr^n(X)=\calhr^n(QX),
$$
for any cofibrant replacement $QX \xrightarrow{\sim} X$. Since $\calhr$ is represented by an $\Omega$-spectrum, it sends weak equivalences between cofibrant objects to isomorphisms, see Theorem~\ref{white}. Hence the definition is independent of the chosen cofibrant replacement. 

Furthermore, if $\calhr$ is represented by the $\Omega$-spectrum $E\in\spece$, Proposition \ref{rectificacion} provides natural isomorphisms
$$
\calhr^n(X)
\cong [QX,E_n]_\epun\cong  [QX,E_n]^\ray,\quad X\in\epun.
$$
\end{rem}

 All of he above can be carried over   to the proper setting.

\begin{definition}\label{cohopropia}
In view of the full embedding~(\ref{fullembed}), consider the subcategory of $\E_*$ defined by
$$
\PP_*=\PP\cap \cofi_*.
$$
That is, its objects are proper spaces (i.e. exterior spaces endowed with the cocompact externology) which are cofibrant, and its morphisms are proper maps preserving the retractive structure over $\ray$.
\end{definition}

\begin{definition}\label{propercohomo} In this context, a \emph{(reduced) proper cohomology theory} is defined as the restriction to $\PP_*$ of a given reduced exterior cohomology theory.
\end{definition}

\begin{proposition}\label{proinva} Any proper cohomology theory is an invariant of proper homotopy under $\ray$.
\end{proposition}

\begin{proof}
By the full embedding~(\ref{fullembed2}), two maps in $\PP_*$ are properly homotopic under $\ray$ if and only if they are exteriorly homotopic under $\ray$. If $Y$ is fibrant, the identification in Proposition \ref{rectificacion}
$$
[X,Y]_{\E_*}\cong [X,Y]^{\ray},
$$
implies that they are homotopic in $\E_*$. The result then follows from the homotopy invariance of exterior cohomology theories.
\end{proof}

\section{From classical cohomology theories}

Our aim in this final section is to illustrate the scope of the general theory developed above by deriving exterior cohomology theories from classical ones. This provides a broad class of examples which shows how  the representability theorem applies in concrete situations and how the resulting theories admit natural interpretations within the framework of proper homotopy theory. The present framework places the main results of \cite{gargarmu3} into a stable context.

Recall, see for instance \cite[\S18.1]{may}, that an (extraordinary) {\em cohomology theory} is a collection of contravariant functors
$$
h^n \colon \cw^{\mathrm{pair}} \longrightarrow \mathbf{Ab},
\qquad n \in \mathbb{Z},
$$
together with natural transformations
$$
\delta \colon h^n(A) \longrightarrow h^{n+1}(X,A),
$$
satisfying the homotopy, exactness, additivity, and excision axioms as formulated in Definition~\ref{noredu}, replacing $\cofi^{\mathrm{pair}}$ by $\cw^{\mathrm{pair}}$ and exterior homotopy by ordinary homotopy. A straightforward exercise shows that our formulation of the excision axiom is equivalent to the one in \emph{op.\ cit.}, namely: if $X$ is the union of subcomplexes $A$ and $B$, then the inclusion
$$
(A,A\cap B)\longrightarrow (X,B)
$$
induces an isomorphism on $h^*$.

On the other hand, see \cite[\S19.2]{may}, an (extraordinary) {\em reduced cohomology theory} is a collection of contravariant functors
$$
\hre^n \colon \cw_* \longrightarrow \mathbf{Ab},
\qquad n \in \mathbb{Z},
$$
satisfying the homotopy, exactness, and additivity axioms appearing in Definition~\ref{redu}, replacing $\cofi_*$ by $\cw_*$ and exterior homotopy relative to $\ray$ by based homotopy. Again, it is easy to check that  the exactness axiom is equivalent to the exactness and suspension axioms as formulated in \emph{op.\ cit.} 

As stated in \cite[\S18,19]{may} a cohomology theory on $\cw^{\mathrm{pair}}$ determines and is determined by a reduced cohomology theory on $\cw_*$.  Furthermore, by standard facts on CW-approximation, these notions descend, respectively, to the corresponding ones of non-reduced and reduced cohomology theories on the categories $\topopair$ and $\topo_*$ of pairs of spaces and well-based spaces. In this general setting, the excision axiom is formulated in terms of excisive triads, and the homotopy axiom is replaced by invariance under weak equivalences in the categories of pairs and based spaces. 

From now on we will also assume that every cohomology theory $h$ considered is {\em continuous}. That is, for every triad of spaces $(X,A,B)$ in which $(A,B)$ is a closed pair,
the natural map
$$
\varinjlim_{(U,V)} h^*(U,V)
\stackrel{\cong}{\longrightarrow}
h^*(A,B)
$$
is an isomorphism, where $(U,V)$ runs through the neighbourhoods of $(A,B)$. 

To give the reader a broader perspective, we mention that Alexander--Spanier cohomology is a continuous cohomology theory, whereas singular cohomology is generally not continuous. On the other hand, \v{C}ech cohomology is continuous but does not, in general, satisfy the exactness axiom required of a cohomology theory on arbitrary spaces. However, on HLC (homologically locally connected) spaces, a hypothesis substantially weaker than the ANR condition, all three theories agree \cite[Chap.~III]{bre}.

We will also need the extension of the Alexandroff compactification to the exterior category:

\begin{definition}\cite[\S1]{garcal} The {\em Alexandroff exterior construction} is  the functor 
$$
(-)^\infty\colon \E\longrightarrow\topo^*
$$
which assigns to an exterior space $X$, the pointed space $X^\infty=X\sqcup\{\infty\}$ with basepoint $\infty$, and extending the topology of $X$ by declaring the sets $\{E\cup\{\infty\},\,E\in\ext\}$ to be open.
\end{definition}

\begin{definition}\cite[Def.~3.1]{gargarmu3}
Let $h$ be a cohomology theory. For each $n\in\Z$, define the contravariant functor
$$
\calhr^n\colon \E \longrightarrow \catab,\quad \calhr^n(X)=\varinjlim_{E\in\ext}h^n(X,E),
$$
where $\ext$ denotes the externology of $X$. Note that  no cofibrancy or
retractability assumptions are required on $X$.
\end{definition}

In \cite{gargarmu3} we established only those properties of these functors required to apply the classical Brown representability theorem. Therefore, some additional verifications are still needed in order to conclude:

\begin{theorem}\label{clasicacoho}
The collection of functors $\{\calhr^n\}_{n\in\Z}$ form a (reduced) exterior cohomology theory.
\end{theorem}

\begin{proof}

On the one hand, there is a natural homeomorphism
$$
(X\timess I)^\infty \cong (X^\infty\times I)/({\infty}\times I),
$$
for every $X\in \E$, see \cite[\S2,3]{garcal}. It follows that the Alexandroff exterior construction preserves homotopies (regardless of whether they are taken relative to $\ray$ or not).

On the other hand, \cite[Thm.~3.1]{gargarmu3} provides a natural isomorphism
\begin{equation}\label{cohomoinf}
\calhr^*(X)\cong h^*(X^\infty,\infty),
\qquad X\in\E,
\end{equation}
for which the continuity hypothesis is necessary.
Combining these two facts, we conclude that $\calhr^*$ is homotopy invariant.

As for exactness, let $ A \hookrightarrow X$ be a cofibration in $\cofi_*$ and consider the following diagram  in $\topo^*$:
$$
\xymatrix{
A^\infty \ar[r] \ar@{^{(}->}[d]
& \br_+^\infty \ar[r] \ar[d]
& {\ast} \ar@{^{(}->}[d]
\\
X^\infty \ar[r]
& (X/A)^\infty \ar[r]
& X^\infty/A^\infty.
}
$$
where the left-hand square is the pushout obtained by applying $(-)^\infty$, which preserves colimits, see \cite[\S2,3]{garcal}, to the corresponding exterior pushout.
As the outer rectangle is clearly a pushout,  the right-hand square is a pushout as well. Thus, since the map $\br_+^\infty\to{\ast}$ is trivially a weak equivalence, left properness of $\topo^*$ implies that $(X/A)^\infty$ is naturally weakly equivalent to $X^\infty/A^\infty$. Therefore, the natural isomorphism \ref{cohomoinf}, together with the exactness of $h$ on the pair $(X^\infty,A^\infty)$, gives the required long exact sequence for $\calhr$ associated to the pair $(X,A)$.

Finally, additivity is proved in \cite[Lemma 4.2]{gargarmu3}.
\end{proof}

For the following recall Definition \ref{propercohomo}:

\begin{corollary}\label{compactsupport}
The proper cohomology theory induced by $\calhr$ is precisely $h_c$, the cohomology with compact support associated to $h$.
\end{corollary}

\begin{proof}
If $X$ is endowed with the cocompact externology, then
$$
\calhr^*(X)
=
\varinjlim_{E\in\ext} h^*(X,E)
=
\varinjlim_{K\,\text{compact}} h^*(X,X\setminus K),
$$
which is precisely the definition of $h_c^*(X)$.
\end{proof}

A quite general instance of the above is given by the following:

\begin{definition} \cite{expari} Let $A$ be a group and $n\ge 0$. An {\em exterior Eilenberg--MacLane} space of type $(A,n)$ is an e-connected exterior space $\calk(A,n)$ such that $\pi^e_m \calk(A,n)=A$ if $m=n$ and zero otherwise. An  {\em exterior Eilenberg--MacLane spectrum} of type $A$ is an $\Omega$-spectrum $E=\{E_n\}_{n\ge 0}$ such that each $E_n$ is an exterior Eilenberg--MacLane space of type of type $(A,n)$.
\end{definition}

\begin{theorem}\label{alex}
Let $h$ denote Alexander--Spanier cohomology with coefficients in an abelian group $G$,
and let $\calhr$ be the associated exterior cohomology theory.
Then the exterior spectrum representing $\calhr$ is, up to stable
equivalence, an exterior Eilenberg--Mac Lane spectrum
of type $G^\infty$, a countable product of copies of $G$. 
\end{theorem}

\begin{proof}
Let
$
E=\{E_n\}_{n\geq 1}
$
be the exterior $\Omega$-spectrum representing $\calhr$ provided by Theorem \ref{repre}. In view of Remark \ref{continuog}, for any $n\ge 1$ and $m\ge 0$ we  have 
$$
\pi_m^e(E_n)
\cong[\bsphc^m,E_n]^\ray
\cong
\calhr^n(\bsphc^m)
$$
which, by construction, agrees with $h^n_c(\bsphc^m;G)$, the compactly
supported Alexander--Spanier  $n$th cohomology of $\bsphc^m$, or compactly supported singular cohomology as $\bsphc^m$ is HLC, with coefficients in $G$. Hence
$$
\pi_m^e(E_n)
\cong
h_c^n(\bsphc^m;G)
\cong
\begin{cases}
G^\infty, & m=n,\\
0, & m\neq n.
\end{cases}
$$
Therefore each $E_n$ is an exterior Eilenberg--Mac Lane space of type
$G^\infty$. Since $E$ is an exterior $\Omega$-spectrum, its adjoint
structure maps are exterior weak equivalences, and thus $E$ is an
exterior Eilenberg--Mac Lane spectrum of type $G^\infty$.
\end{proof}

Observe, furthermore, that each $E_n$ has only one ray  up to exterior homotopy, that is, one {\em exterior strong end}. Indeed, by representability,
$$
[\ray,E_n]\cong \calhr^n(\ray)=0,
$$
as every reduced exterior cohomology
theory vanishes on the zero object.

\bigskip\bigskip\bigskip\bigskip

\noindent{\sc Departamento de \'Algebra, Geometr\'{\i}a y Topolog\'{\i}a,
Facultad de Ciencias,
Universidad de M\'alaga,
Blvr. Louis Pasteur 31,
29010 M\'alaga, Spain.}

\noindent
\texttt{aceres@uma.es}

\noindent
\texttt{aniceto@uma.es}

\bigskip

\noindent{\sc Departamento de Matem\'aticas, Estad{\'\i}stica e Investigaci\'on Operativa,
Universidad de La Laguna,
Av.~Astrof{\'\i}sico Francisco S\'anchez S/N,
38206 La Laguna, Spain.}

\noindent
\texttt{jmgarcal@ull.edu.es}

\end{document}